\documentclass[a4paper]{article}

\usepackage{amsthm, amsmath, amsfonts, amssymb, accents}
\usepackage{graphicx}
\usepackage{srcltx}

\usepackage{dsfont}

\newtheorem{lemma}{Lemma}[section]
\newtheorem{theorem}{Theorem}[section]

\theoremstyle{remark}
\newtheorem{remark}{Remark}[section]

\numberwithin{equation}{section}

\def\Om{\Omega}
\def\om{\omega}
\def\e{\varepsilon}
\def\g{\gamma}

\def\l{\lambda}
\def\p{\partial}

\def\a{\alpha}
\def\b{\beta}

\def\d{\delta}
\def\L{\Lambda}
\def\z{\zeta}

\def\Odr{\mathcal{O}}
\def\H{W_2}

\def\Hinf{W_\infty}

\def\di{\,\mathrm{d}}

\def\I{\mathrm{I}}
\def\iu{\mathrm{i}}

\def\He{\mathcal{H}_\a^\e}
\def\he{\mathfrak{h}_\a^\e}
\def\her{\mathfrak{h}_{\a,\mathrm{r}}^\e}
\def\hei{\mathfrak{h}_{\a,\mathrm{i}}^\e}

\def\Ho{\mathcal{H}^0_\a}
\def\ho{\mathfrak{h}_0^0}

\def\bb{\mathfrak{b}^\e}

 \DeclareMathOperator{\RE}{Re}
\DeclareMathOperator{\IM}{Im} \DeclareMathOperator{\spec}{\sigma}

\DeclareMathOperator{\dist}{dist} 

\DeclareMathOperator{\Dom}{\mathcal{D}}

\begin{document}
\allowdisplaybreaks

\title{\textbf{Discrete spectrum of \\ thin $\mathcal{PT}$-symmetric waveguide}}
\author{Denis Borisov}
\date{\small
\begin{center}
\begin{quote}
{\it
 Institute of Mathematics CS USC RAS, Chernyshevsky str. 112, Ufa, 450008, Russia}
 \\
   \smallskip
{\it Bashkir State Pedagogical University, October St.~3a,
Ufa, 450000,  Russia}
\\
\smallskip
{\it E-mail:} \texttt{borisovdi@yandex.ru} \\
\smallskip{\it Homepage:} \texttt{http://borisovdi.narod.ru/}
\end{quote}
\end{center}}

\maketitle

\begin{abstract}
In a thin multidimensional layer we consider a second order differential $\mathcal{PT}$-symmetric operator. The operator is of rather general form and its coefficients are arbitrary functions depending both on slow and fast variables. The $\mathcal{PT}$-symmetry of the operator is ensured by the boundary conditions of Robin type with pure imaginary coefficient. In the work we determine the limiting operator, prove the uniform resolvent convergence of the perturbed operator to the limiting one, and derive the estimates for the rates of convergence. We establish the convergence of the  spectrum of perturbed operator to that of the limiting one. For the perturbed eigenvalues converging to the limiting discrete  ones we prove that they are real and construct their complete asymptotic expansions. We also obtain the complete asymptotic expansions for the associated eigenfunctions.
\end{abstract}

\begin{quote}
\\

\end{quote}

\section{Introduction}

In the end of the last century a new direction in mathematical physics appeared devoted to $\mathcal{PT}$-symmetric operators. This notion indicates usually differential (or more general) operators commuting with the composition $\mathcal{PT}$, where $\mathcal{T}$ is the operator of complex conjugation,  $(\mathcal{T}u)(x)=\overline{u(x)}$, and $\mathcal{P}$ is a some operator describing symmetric transformation w.r.t. spatial variable, say, $(\mathcal{P}u)(x)=u(-x)$. Such operators are usually non-self-adjoint and the main interest is usually related with their various spectral properties. One of first pioneering works initiated a impetuous study of
$\mathcal{PT}$-symmetric operators are the papers \cite{BB}--\cite{Z4}, see also the survey \cite{B} as well as the references in the cited works.

One of the most interesting properties of $\mathcal{PT}$-symmetric operators is the fact that they can possess real spectrum that gives a chance for quantum mechanical interpretation of these operators. In particular, there was found a series of  $\mathcal{PT}$-symmetric operators with real spectra, see, for instance, \cite{Caliceti-Cannata-Graffi_2006}--\cite{Znojil_2001}. It should be stressed that the most part of studies was devoted to the case of Schr\"odinger operator with $\mathcal{PT}$-symmetric potential.

A more complicated model of  $\mathcal{PT}$-symmetric waveguide where the $\mathcal{PT}$-symmetry was originated by boundary conditions and not by the differential expression was suggested in work \cite{BK1}. Here there was considered the Laplace operator in an infinite straight strip with Robin boundary condition. The coefficient in the boundary condition was pure imaginary that ensures the required
$\mathcal{PT}$-symmetry. It was assumed that the coefficient differs to a constant just by a finite function multiplied by a small parameter. The essential spectrum of such operator was found and it happened to be a fixed real semi-axis. The phenomenon of new eigenvalues emerging from the threshold of essential spectrum was studied.

A series of numerical experiments performed in work  \cite{KT} showed that in the case when the aforementioned small parameter becomes finite and increases, in the spectrum of this model there appear also pairs of complex conjugate isolated eigenvalues that behave in a quite fanciful way. A similar model, but much more complicated, of Laplace-Beltrami operator in a strip on a two-dimensional Riemann manifold was considered in \cite{KS}. There was obtained a series of general results on the operator and its spectrum.

The described model in work \cite{BK1} was in fact an operator with a small regular perturbation that simplified essentially the studying. More complicated cases of singularly perturbed $\mathcal{PT}$-symmetric operators were considered in recent works \cite{IMM2012}, \cite{AsAn2012}. In \cite{IMM2012} a model from \cite{BK1} was again studied but the coefficient in the boundary condition was a sufficiently smooth bounded function and the perturbation was cutting out two small symmetric holes inside the strip. The limiting operator here is the same $\mathcal{PT}$-symmetric operator but without small holes.
The uniform resolvent convergence of the perturbed operator to the limiting one was proven and the estimates for the rates of convergence were established. Moreover, we studied in details the phenomenon of new eigenvalues emerging from the threshold of the essential spectrum. It was shown that here the necessary and sufficient conditions for existence and absence of such eigenvalues differ substantially from similar results for self-adjoint operators \cite{N}.

In work \cite{AsAn2012} one more extension of the model in \cite{BK1} was studied. Here the strip was replaced by a multi-dimensional layer and the singularity of the perturbation was a small width of the layer.  The main result of paper \cite{AsAn2012} is determination of the limiting operator for such model, proving the uniform resolvent convergence of the perturbed operator to the limiting one and establishing the estimates for the rates of convergence. The limiting operator happened to be self-adjoint and it allowed us to state that even if the spectrum of the perturbed operator is not real, at least, it is located near the real axis.

The present work is devoted to generalization and further developing of the results of work \cite{AsAn2012}. We again consider
$\mathcal{PT}$-symmetric operator in a thin multi-dimensional layer. But in contrast to \cite{AsAn2012}, we consider an arbitrary second order scalar operator with variable coefficients, not just the Laplacian. For the coefficients of the operator we impose rather weak smoothness conditions as well as the conditions ensuring  $\mathcal{PT}$-symmetry. Moreover, these coefficients can depend on fast (rescaled) variable in the transversal direction in the layer that in fact make these coefficients fast oscillating.
$\mathcal{PT}$-symmetry of the operator is again originated
by Robin condition with pure imaginary coefficient.

The first part of the work is devoted to determining the form of the limiting operator. Such operator is found, and it is essentially more complicated in comparison with  \cite{AsAn2012}. It is related to the presence of all coefficients in the perturbed operator and rather nontrivial formulae for the coefficients of limiting operator. Here our main result is the proof of the uniform resolvent convergence of the perturbed operator to the limiting one and establishing the estimates for the rate of convergence. It is shown that the order of these estimates is the best possible, while in
\cite{AsAn2012} such result was absent.

In the second part of the work we consider the asymptotic behavior of the spectrum of the perturbed operator. We first prove the convergence of the spectrum of perturbed operator to that of the limiting operator. It should be stressed that here we can not employ the classical theorems on spectrum convergence since the perturbed operator is not self-adjoint. Instead of this we propose an approach based on a non-self-adjoint version of Birman-Schwinger principle suggested  in \cite{G}, \cite{MSb} and we combine it the proven uniform resolvent convergence.

Then we study the behavior of perturbed eigenvalues converging to isolated limiting eigenvalues. Here we succeeded to find a simple but original trick and to show that all such perturbed eigenvalues are real no matter what the multiplicity is  (\ref{5.19}), (\ref{5.21}), (\ref{5.20}). We note that similar results on reality of considered eigenvalues in \cite{BK1}, \cite{IMM2012} were based on their simplicity.

We also construct the complete asymptotic expansions of aforementioned eigenvalues and associated eigenfunctions. The asymptotics are constructed first formally on the basis of multiscale method \cite{MS} and then they are justified. And if the formal construction does not really differ from similar constructions for self-adjoint operators in thin domains (see, for instance, \cite{NazB}--\cite{ESAIM}, as well as \cite{CCP}, \cite{PP}), we failed to apply the standard justification technique from the self-adjoint case \cite{NazB}, \cite{Izv2003}. Here we have to develop certain functional approach for the justification. It should be also stressed that no results on asymptotic behavior of the spectrum were obtained in  \cite{AsAn2012}.

In conclusion let us describe the structure of the paper. In the next section we formulate the problem and state the main results. In the third section we prove general qualitative properties of the perturbed operator. The forth section is devoted to proving the uniform resolvent convergence and obtaining the estimates for the rate of convergence. In the fifth section we prove the convergence of the spectrum. In the sixth section we construct formally the asymptotic expansions for the eigenvalues and the eigenfunctions of the perturbed operator, while in the seventh section they are rigourously justified.

\section{Statement of problem and main results}

Let $x=(x',x_n)$ be Cartesian coordinates in $\mathds{R}^n$, $n\geqslant 2$, $\Om^\e:=\{x: -\e/2<x_n<\e/2\}$ be a thin multi-dimensional layer in $\mathds{R}^n$, $\e$ be a small positive parameter and $\e\leqslant \e_0$, where $\e_0$ is a small fixed number. We denote
\begin{equation*}
\Om:= \{(x',\xi):\ x'\in \mathds{R}^{n-1},\ \xi\in(-1/2,1/2)\},\quad \Pi:=\Om\times(0,\e_0).
\end{equation*}
In $\Pi$ we define functions $A_{ij}=A_{ij}(x',\xi,\e)$, $A_j=A_j(x',\xi,\e)$, $A_0=A_0(x',\xi,\e)$ satisfying the conditions
\begin{gather}
A_{ij}(\cdot,\cdot,\e), A_j(\cdot,\cdot,\e) \in C^1(\overline{\Om}),
\quad
A_0(\cdot,\cdot,\e)\in C(\overline{\Om}),
\label{1.1}
\\
A_{ji}=A_{ij}, \quad \sum\limits_{i,j=1}^{n} A_{ij}(x',\xi,\e)\z_i\z_j\geqslant c_0|\z|^2,\quad \z\in\mathds{R}^n,\quad (x',\xi,\e)\in\overline{\Pi},
\label{1.2}
\end{gather}
where $c_0$ is a positive constant independent of $x'$, $\xi$, $\e$, and $\z$. Functions $A_{ij}$ are assumed to be real-valued, while functions $A_j$, $A_0$ are complex-valued and
\begin{equation*}
A_{ij}, \ \nabla_{x,\xi} A_{ij}, \ A_j, \ \nabla_{x,\xi} A_j,\ A_0 \in L_\infty(\Pi).
\end{equation*}
Moreover, we suppose the following symmetry conditions
\begin{equation}\label{1.3}
\begin{aligned}
&A_{ij}(x',-\xi,\e)=A_{ij}(x',\xi,\e),  &&  A_{in}(x',-\xi,\e)=-A_{in}(x',\xi,\e), && i,j=1,\ldots,n-1,
\\
&A_{nn}(x',-\xi,\e)=A_{nn}(x',\xi,\e), && \overline{A}_j(x',-\xi,\e)=A_j(x',\xi,\e), && j=1,\ldots,n-1,
\\
&A_n(x',-\xi,\e)=-\overline{A}_n(x',\xi,\e), &&
A_0(x',-\xi,\e)=\overline{A}_0(x',\xi,\e).
\end{aligned}
\end{equation}

Let $a=a(x',\e)$ be a real function belonging to  $\Hinf^1(\mathds{R}^{n-1})$ for each ${\e\in[0,\e_0]}$.
We indicate
\begin{align*}
\eta(\e):=&\sum\limits_{i,j=1}^{n} \sup\limits_{\overline{\Om}} |A_{ij}(x',\xi,\e)-A_{ij}(x,\xi,0)|
+\sum\limits_{i,j=1}^{n} \sup\limits_{\overline{\Om}} |\nabla_{x,\xi}(A_{ij}(x',\xi,\e)-A_{ij}(x',\xi,0))|
\\
&+\sum\limits_{j=1}^{n} \sup\limits_{\overline{\Om}} |A_j(x',\xi,\e)-A_j(x',\xi,0)|
+ \sup\limits_{\mathds{R}^{n-1}} |\a(x',\e)-\a(x',0)|.
\end{align*}
In what follows functions $A_{ij}$, $A_j$, $A_0$, $\a$ are assumed to be continuous w.r.t. $\e$ at the point $\e=0$, namely,
\begin{equation}\label{1.6}
\lim\limits_{\e\to+0} \eta(\e)=0.
\end{equation}
We let
\begin{align*}
&A_{ij}^\e(x):=A_{ij}\left(x',\frac{x_n}{\e},\e\right), &&
A_j^\e(x):=A_j\left(x',\frac{x_n}{\e},\e\right),
\\
&A_0^\e(x):=A_0\left(x',\frac{x_n}{\e},\e\right), && \a^\e(x'):=\a(x',\e).
\end{align*}

The main object of the study in the present work is the operator
\begin{equation}\label{1.4}
\begin{aligned}
\He=-\sum\limits_{i,j=1}^{n} \frac{\p\hphantom{x}}{\p x_i} A_{ij}^\e \frac{\p\hphantom{x}}{\p x_j} + \sum\limits_{j=1}^{n} \left(A_j^\e\frac{\p\hphantom{x}}{\p x_j}-\frac{\p\hphantom{x}}{\p x_j} \overline{A_j^\e}\right)+ A_0^\e \quad\text{в}\quad \Om^\e
\end{aligned}
\end{equation}
subject to the boundary condition
\begin{equation}\label{1.5}
\left(\frac{\p\hphantom{\nu}}{\p\nu^\e}+\iu\a\right)u=0 \quad\text{на}\quad\p\Om^\e, \qquad \frac{\p\hphantom{\nu }}{\p\nu^\e}:=\sum\limits_{j=1}^{n} A_{nj}^\e\frac{\p}{\p x_j}+\overline{A_n^\e},
\end{equation}
where $\iu$ is the imaginary unit.

Rigourously we introduce operator $\He$ as that in $L_2(\Om^\e)$ defined by the differential expression (\ref{1.4}) on the domain
\begin{equation}\label{1.7}
\Dom(\He) = \{u\in \H^2(\Om^\e): \ \text{condition (\ref{1.5}) is satisfied}\}.
\end{equation}
In what follows we call this operator perturbed.

The main aim of the work is to study the asymptotic behavior
of the resolvent and the discrete spectrum of operator $\He$ as $\e\to+0$.

To formulate the main results we shall need additional notations. In $L_2(\Om^\e)$ we define the mappings
\begin{equation}\label{1.8}
(\mathcal{P}u)(x):=u(x',-x_n),\quad \mathcal{T}u:=\overline{u}. \end{equation}
Our first result describe the qualitative properties of operator $\He$.

\begin{theorem}\label{th2.1}
Operator $\He$ is $m$-sectorial, $\mathcal{T}$-self-adjoint and $\mathcal{P}$-pseudo-Hermitian, i.e.,
\begin{equation}\label{1.9}
(\He)^*=\mathcal{T}\He\mathcal{T},\quad (\He)^*=\mathcal{P}\He\mathcal{P},
\end{equation}
and $\mathcal{PT}$-symmetric
\begin{equation}\label{1.10}
\mathcal{PT}\He=\He\mathcal{PT}.
\end{equation}
The adjoint operator for $\He$ is given by the identity
\begin{equation}\label{1.11}
(\He)^*=\mathcal{H}_{-\a}^\e.
\end{equation}
The spectrum of operator $\He$ satisfies the inclusion
\begin{equation}\label{1.12}
\begin{gathered}
\spec(\He)\subseteq\mathds{K},
\\
\begin{aligned}
\mathds{K}:&=\left\{z\in\mathds{C}:\ |\IM\l|\leqslant
\frac{c_3}{c_0}\left(c_1+\sqrt{c_1^2+c_0(|\RE\l|+c_2)}\right)+c_2
\right\}
\\
&\subseteq\left\{z\in\mathds{C}:\ |\IM\l|\leqslant
 \frac{c_3}{\sqrt{c_0} }\sqrt{|\RE z|}+\frac{(c_1+\sqrt{c_1^2+c_0 c_2})c_3}{c_0}+c_2\right\},
\end{aligned}
\end{gathered}
\end{equation}
where
\begin{gather*}
c_1:= \left(\sum\limits_{j=1}^{n} \sup\limits_{\Pi}{}^2 |A_j(x',\xi,\e)|\right)^{1/2},
\\
c_2:=\sup\limits_{\Pi}{}^2 |A_0(x',\xi,\e)|,
\quad
c_3:=2\sup\limits_{\mathds{R}^{n-1}\times [0,\e_0]} |\a(x',\e)|.
\end{gather*}
\end{theorem}

To describe the asymptotic behavior of the resolvent for operator $\He$, we introduce the limiting operator. Let
\begin{equation}\label{1.13}
\begin{aligned}
&A_{ij}^0(x'):=\int\limits_{-1/2}^{1/2} \left(A_{ij}(x',\xi,0)- \frac{A_{in}(x',\xi,0)A_{nj}(x',\xi,0)}{A_{nn}(x',\xi,0)}
\right)\di\xi,
\\
&A_j^0(x'):=\int\limits_{-1/2}^{1/2} \left(A_j(x',\xi,0)- \frac{A_n(x',\xi,0)A_{nj}(x',\xi,0)}{A_{nn}(x',\xi,0)}\right)\di\xi,
\\
&A_0^0(x'):=\int\limits_{-1/2}^{1/2} \left(A_0(x',\xi,0)+ \frac{\a^2(x',0)}{A_{nn}(x',\xi,0)}\right.
\\
&\hphantom{A_0^0(x'):=\int\limits_{-1/2}^{1/2}9}\left.  -2\iu \a(x',0) \frac{\RE A_n(x',\xi,0)}{A_{nn}(x',\xi,0)}-\frac{|A_n(x',\xi,0)|^2}{A_{nn}(x',\xi,0)}
\right)\di\xi,
\end{aligned}
\end{equation}
where $i,j=1,\ldots,n-1$.
In $L_2(\mathds{R}^{n-1})$ we define the operator
\begin{equation}\label{1.14}
\Ho:=-\sum\limits_{i,j=1}^{n-1}\frac{\p\hphantom{x}}{\p x_i} A_{ij}^0 \frac{\p\hphantom{x}}{\p x_j}  + \sum\limits_{j=1}^{n-1} \left( A_j^0 \frac{\p\hphantom{x}}{\p x_j} - \frac{\p\hphantom{x}}{\p x_j} \overline{A_j^0}\right) + A_0^0
\end{equation}
on domain $\H^2(\mathds{R}^{n-1})$. By $\mathcal{Q}^\e$ we denote the projector in $L_2(\Om^\e)$
\begin{equation}\label{1.15}
(\mathcal{Q}^\e f)(x'):=\e^{-1}\int\limits_{-\e/2}^{\e/2} f(x)\di x_n
\end{equation}
and we let
\begin{equation}\label{1.16}
L^\e:=\mathcal{Q}^\e L_2(\Om^\e),\quad L_\bot^\e:=\mathcal{Q}_\bot^\e L_2(\Om_\e),
\quad \mathcal{Q}_\bot^\e:=\I-\mathcal{Q}^\e.
\end{equation}
Space $L_2(\Om^\e)$ can be represented as the direct sum
\begin{equation}\label{1.17}
L_2(\Om^\e)=L^\e\oplus L_\bot^\e.
\end{equation}
In the sense of this expansion, the operator $\e^{-1/2} (\mathcal{H}_\a^0-\l)^{-1}\mathcal{Q}^\e$ acting in $L^\e$ for appropriate $\l\in\mathds{C}$ can be extended to the operator  $\e^{-1/2} (\mathcal{H}_\a^0-\l)^{-1}\mathcal{Q}^\e\oplus 0$ acting in $L_2(\Om^\e)$.

Let us formulate the main result on asymptotic behavior of the resolvent for operator $\mathcal{H}_\a^\e$.

\begin{theorem}\label{th2.2}
Operator $\mathcal{H}_\a^0$ is self-adjoint. For each $\l\in \mathds{C}\setminus(\mathds{K}\cup\spec(\Ho))$ and sufficiently small $\e$ the operators  $(\He-\l)^{-1}$ and $(\mathcal{H}_\a^0-\l)^{-1}$ are well-defined and bounded. For each $f\in L_2(\Om^\e)$ the uniform in $\e$ and $f$ estimates
\begin{equation}
\big\|(\mathcal{H}_\a^\e-\l)^{-1}f-\big(  (\mathcal{H}_\a^0-\l)^{-1}\mathcal{Q}^\e\oplus 0\big)f
\big\|_{L_2(\Om^\e)}\leqslant (\e+\eta(\e)) C(\l) \|f\|_{L_2(\Om^\e)}
\label{1.18}
\end{equation}
and
\begin{equation}\label{1.19}
\begin{aligned}
&\big\|(\mathcal{H}_\a^\e-\l)^{-1}f-\big( (\mathcal{H}_\a^0-\l)^{-1}\mathcal{Q}^\e\oplus 0\big)f
\\
&\hphantom{\big\|(\mathcal{H}_\a^\e-\l)^{-1}f}-\e \mathcal{W}^\e (\mathcal{H}_\a^0-\l)^{-1}\mathcal{Q}^\e f
\big\|_{\H^1(\Om^\e)}\leqslant  (\e+\eta(\e)) C(\l) \|f\|_{L_2(\Om^\e)}
\end{aligned}
\end{equation}
hold true, where  $C(\l)$ are some constants independent of $\e$, $f$ but dependent of $\l$, and the operator $\mathcal{W}^\e: \H^2(\mathds{R}^{n-1})\to \H^1(\Om^\e)$ is determined as
\begin{align*}
(\mathcal{W}^\e u)(x,\e):=&
-\sum\limits_{i,j=1}^{n-1} \frac{\p u}{\p x_j}(x') \int\limits_{0}^{\frac{x_n}{\e}} \frac{A_{nj}(x',t,\e)}{A_{nn}(x',t,\e)}\di t - u(x') \int\limits_{0}^{\frac{x_n}{\e}} \overline{A}_n(x',t,\e)\di t
\\
& -\iu \a^\e(x') u(x') \int\limits_{0}^{\frac{x_n}{\e}} \frac{\di t}{A_{nn}(x',t,\e)}.
\end{align*}
\end{theorem}

We discuss briefly the results of this theorem. We first note that in comparison with the particular case in \cite{AsAn2012} the coefficients of the limiting operator are rather nontrivial and in fact we have ``mixing'' of the coefficients of the perturbed operator, see (\ref{1.13}).

It should be stressed independently that the estimates for the rates of convergence in Theorem~\ref{th2.2} have the best possible order. Namely, if the coefficients of the perturbed operator and function $f$ are infinitely differentiable so that $\eta(\e)=C\e$, $C=const$, on the basis of multiscale method  \cite{MS} one can construct the complete asymptotic expansion for the function $(\He-\l)^{-1}f$. In what follows we employ similar anz\"atz to construct the asymptotics for the eigenvalues.  In our case they lead to the formulae
\begin{equation*}
(\He-\l)^{-1}f=\big((\Ho-\l)^{-1}\oplus 0\big)f+\e u_1 +\Odr(\e^2+\eta^2(\e))
\end{equation*}
in the norm of $L_2(\Om^\e)$ and
\begin{equation*}
(\He-\l)^{-1}f=\big((\Ho-\l)^{-1}\oplus 0\big)f-\e \mathcal{W}^\e (\mathcal{H}_\a^0-\l)^{-1}\mathcal{Q}^\e f+\e^2 u_2 +\Odr(\e^2)
\end{equation*}
in the norm of $\H^1(\Om^\e)$, where $u_1=u_1(x',x_n\e^{-1})$, $u_2=u_2(x',x_n\e^{-1})$ are some functions depending on the choice $f$. Exactly this fact implies the optimality of the estimates for the rate of convergence in Theorem~\ref{th2.2}.

Our next result describes the convergence of the perturbed spectrum. We stress that in our case we can not employ the classical theorems on spectrum convergence since the perturbed operator $\He$ is not self-adjoint. Moreover, the perturbed and the limiting operator act on different spaces and for the perturbed operator this space also depends on $\e$.

\begin{theorem}\label{th2.3}
The spectrum of operator $\He$ converges to that of operator $\Ho$ as $\e\to+0$. Namely, for each compact set $\mathds{M}\subset \mathds{C}$, $\d>0$ there exist $\kappa(\mathds{M},\d)>0$ such that for $0<\e<\kappa(\mathds{M},\d)$ the part of the spectrum  $\spec(\He)\cap \mathds{M}$ of operator $\He$ is located in the $\d$-neighborhood of the part of the spectrum $\spec(\Ho)\cap\mathds{M}$ of operator $\Ho$. If $\l_0$ is an isolated $m$-multiple eigenvalues of operator $\Ho$, then there exist exactly $m$ eigenvalues of operator $\He$ taken counting multiplicity and converging to $\l_0$ as $\e\to+0$. For sufficiently small $\e$ these eigenvalues are real.
\end{theorem}

This theorem states the spectrum convergence in each compact part of the complex plane. At that, it does not exclude the existence of points in the spectrum tending to infinity as
$\e\to+0$.  The statement on reality of eigenvalues converging to isolated limiting eigenvalues deserves special attention. The only requirement is the finite multiplicity of the limiting eigenvalue, and generally speaking, the perturbed eigenvalues converging to this limiting one are not necessary simple. This is an advantage of our result in comparison with similar statements in   \cite{BK1}, \cite{IMM2012}, where the simplicity of the eigenvalues was the basis for the proof of the reality.

Our next result is devoted to obtaining the complete asymptotic expansions of the perturbed operator converging to limiting eigenvalues of finite multiplicity as well as to obtaining the complete asymptotic expansions for the associated eigenfunctions.

\begin{theorem}\label{th2.4}
Suppose that the functions $A_{ij}$, $A_j$, $A_0$, $a$ are independent of $\e$, infinitely differentiable w.r.t. $x$ so that for each $\b\in \mathds{Z}_+^{n-1}$
\begin{equation}\label{1.20}
\frac{\p^{|\b|} A_{ij}}{\p {x'}^\b},\ \frac{\p^{|\b|} A_j}{\p {x'}^\b},\  \frac{\p^{|\b|} A_0}{\p {x'}^\b},\ \frac{\p^{|\b|} \a}{\p {x'}^\b}\in C^2(\overline{\Om})\cap L_\infty(\overline{\Om}).
\end{equation}
Let $\l^0$ be an isolated $m$-multiple eigenvalue of operator $\Ho$. Then asymptotic expansions of eigenvalues $\l_k^\e$, $k=1,\ldots,m$, converging to $\l^0$ as $\e\to+0$ read as
\begin{equation}\label{1.22}
\l_k^\e=\l^0+\sum\limits_{p=1}^{\infty} \e^p \L_k^{(p)},
\end{equation}
where numbers $\L_k^{(p)}$ are defined in the sixth section. If there exists $r>0$ such that for all sufficiently small $\e$
\begin{equation}\label{1.23}
|\l_k^\e-\l_j^\e|\geqslant C\e^r,\quad k\not=j,
\end{equation}
where $C$ is a constant independent of $\e$, $k$, $j$, then the eigenfunctions associated with $\l_k^\e$ can be chosen so that they satisfy the asymptotics
\begin{equation}\label{1.24}
\psi_k^\e(x)=\e^{-1/2} \left(\phi_k(x')+\sum\limits_{p=1}^{\infty} \e^p\phi_k^{(p)}(x',\xi)\right)
\end{equation}
in the norm of $\H^1(\Om^\e)$, where the terms of this series are determined in the sixth section.
\end{theorem}

It should be noticed that the independence of the functions
$A_{ij}$, $A_j$, $A_0$, $\a$ of $\e$ is inessential and is made just for simplicity. In the case the coefficients depend on $\e$, the construction remains the same and all the formulae in the sixth section are kept unchanged, one just need to assume that the coefficients of the asymptotics depend on
$\e$ and this dependence is originated by the same dependence for $A_{ij}$, $A_j$, $A_0$, $\a$. At the same time, to justify the asymptotic expansions one has to suppose the uniform $\e$ boundedness in the norm of $L_\infty(\Om)$ for all the derivatives in (\ref{1.20}). In the case, if the coefficients $A_{ij}$, $A_j$, $A_0$, $\a$ are expanded into asymptotic series w.r.t. $\e$, then one can substitute these expansions into the formulae for the coefficients of asymptotics (\ref{1.22}), (\ref{1.24}), obtain similar expansions for the coefficients and substitute them into series (\ref{1.22}), (\ref{1.24}). The obtained double asymptotic series are the asymptotics for the eigenvalues and the eigenfunctions of the perturbed operator. The described constructions, based on simple ideas,  are rather cumbersome and bulky from technical point of views. This is the reason why we do not provide these calculations in the work restricting ourselves just to the case when functions $A_{ij}$, $A_j$, $A_0$ are independent of $\e$.


\section{Qualitative properties of operator $\He$}

In the present section we prove Theorem~\ref{th2.1}. The main ideas are borrowed from \cite[\S 3]{BK1}.

The proof is based on the theory of sectorial sesquilinear forms, cf. \cite[Гл. V\!I]{K}. In the space $L_2(\Om^\e)$ we define sesquilinear form
\begin{align}
&
\begin{aligned}
\he(u,v):=&\sum\limits_{i,j=1}^{n} \left(A_{ij}^\e\frac{\p u}{\p x_i},\frac{\p v}{\p x_j}\right)_{L_2(\Om^\e)}+ \sum\limits_{j=1}^{n}  \left(A_j^\e\frac{\p u}{\p x_j},v\right)_{L_2(\Om^\e)}
\\
&+ \sum\limits_{j=1}^{n} \left(u,A_j^\e\frac{\p v}{\p x_j}\right)_{L_2(\Om^\e)}
+(A_0^\e u,v)_{L_2(\Om^\e)}+\iu \bb(\a^\e u,v),
\end{aligned}\label{3.1}
\\
&
\bb(u,v):=\int\limits_{\mathds{R}^{n-1}}  u\left(x',\frac{\e}{2}\right) \overline{v}\left(x',\frac{\e}{2}\right)\di x'
-  \int\limits_{\mathds{R}^{n-1}} u\left(x',-\frac{\e}{2}\right) \overline{v}\left(x',-\frac{\e}{2}\right)\di x',\nonumber
\end{align}
on the domain $\Dom(\he):=\H^1(\Om^\e)$. Here $\bb$ is to be treated as a sesquilinear form in  $L_2(\Om^\e)$ on the domain $\H^1(\Om^\e)$.

According to \cite[Гл. V\!I, \S 1.1]{K}, the real and imaginary parts of form $\he$ read as
\begin{equation}\label{3.2}
\begin{aligned}
\her(u,v):=&\sum\limits_{i,j=1}^{n} \left(A_{ij}^\e\frac{\p u}{\p x_i},\frac{\p v}{\p x_j}\right)_{L_2(\Om^\e)}
+  \sum\limits_{j=1}^{n} \left( \left(A_j^\e\frac{\p u}{\p x_j},v\right)_{L_2(\Om^\e)} + \left(u, A_j^\e\frac{\p v}{\p x_j}\right)_{L_2(\Om^\e)} \right)
\\
&+\frac{1}{2}\left((A_0^\e u,v)_{L_2(\Om^\e)} + (u, A_0^\e v)_{L_2(\Om^\e)}\right)
\end{aligned}
\end{equation}
and
\begin{equation}\label{3.3}
\hei(u,v):=
\frac{1}{2\iu}\left((A_0^\e u,v)_{L_2(\Om^\e)} - (u, A_0^\e v)_{L_2(\Om^\e)}\right)+\bb(\a^\e u,v).
\end{equation}
The domain of these forms is again the space $\H^1(\Om^\e)$.

It is clear  that form $\her$ is densely defined, symmetric, and closed. The elementary estimates
\begin{equation}\label{3.4}
\begin{aligned}
 &\left|
\sum\limits_{j=1}^{n} \left(A_j^\e\frac{\p u}{\p x_j},u\right)_{L_2(\Om^\e)} + \left(u, A_j^\e\frac{\p u}{\p x_j}\right)_{L_2(\Om^\e)}
\right|
\\
&\hphantom{\sum\limits_{j=1}^{n}11}\leqslant  \sum\limits_{j=1}^{n} \sup\limits_{\overline{\Om}^\e} |A_j^\e| \Big\|\frac{\p u}{\p x_j}\Big\|_{L_2(\Om^\e)}\|u\|_{L_2(\Om^\e)}
\leqslant \d\|\nabla u\|_{L_2(\Om^\e)}^2+ c_1\d^{-1}\|u\|_{L_2(\Om^\e)}^2,
\end{aligned}
\end{equation}
$\d$ is arbitrary, and
\begin{equation}\label{3.5}
\left|((\RE A_0^\e)u,u)_{L_2(\Om^\e)}\right|\leqslant c_2\|u\|_{L_2(\Om^\e)}^2
\end{equation}
imply the lower semi-boundedness of form $\her$:
\begin{equation}\label{3.6}
\her(u,u)\geqslant -\frac{c_1+c_2c_0}{c_0} \|u\|_{L_2(\Om^\e)}^2.
\end{equation}

It is easy to see that boundary term in form $\hei$ can be estimated from below as
\begin{equation}\label{3.7}
\left|\bb(\a^\e u,u)\right|
=\left|\int\limits_{\Om^\e} \a^\e(x')\frac{\p|u|^2}{\p x_n}\di x\right|\leqslant c_3\|\nabla u\|_{L_2(\Om^\e)}\|u\|_{L_2(\Om^\e)}.
\end{equation}
Together with estimate (\ref{3.4}) for $\d=c_0/2$ and for arbitrary $\d$ it follows that form $\hei$ is relatively bounded w.r.t. form $\her$. Namely, for each $\d>0$ the estimate
\begin{equation}\label{3.8}
|\hei(u,u)|\leqslant \d|\her(u,u)|+C(\d) \|u\|_{L_2(\Om^\e)}^2
\end{equation}
holds true, where $C(\d)$ is a constant independent of  $u$.
By \cite[Гл. V\!I, \S 1, Теорема 1.33]{K}, the obtained properties of forms $\her$ and $\hei$ yield that form $\he$ is sectorial. Applying then first representation theorem \cite[Гл. V\!I, \S 2.1, Теорема 2.1]{K}, we conclude that there exists  $m$-sectorial operator $\widetilde{\mathcal{H}}^\e_\a$ such that \begin{equation}\label{3.9}
\he(u,v)=(\widetilde{\mathcal{H}}^\e_\a u,v)_{L_2(\Om^\e)}
\end{equation}
for each $u\in\Dom(\widetilde{\mathcal{H}}^\e_\a)$, $v\in\Dom(\he)$. The domain of operator $\widetilde{\mathcal{H}}^\e_\a$ consists of functions $u\in\Dom(\he)$ such that there exists a function $f\in L_2(\Om^\e)$ depending on the choice $u$  and satisfying
identity
\begin{equation}\label{3.10}
\he(u,v)=(f,v)_{L_2(\Om^\e)}
\end{equation}
for each $v\in \Dom(\he)$. It is easy to make sure that $\Dom(\He)\subseteq\Dom(\widetilde{\mathcal{H}}^\e_\a)$ and operator $\widetilde{\mathcal{H}}^\e_\a$ is the extension of operator $\He$. To prove  $m$-{sectoriality} of the operator, now it is sufficient to check the identity $\He=\widetilde{\mathcal{H}}^\e_\a$ that is equivalent to the coinciding of domains. On the other hand, the latter is equivalent to the fact that each solution of integral identity (\ref{3.10}) for each $f\in L_2(\Om^\e)$ belongs to $\Dom(\He)$. This fact means the validity of smoothness improving theorems for generalized solutions of elliptic boundary value problems. In our case such theorem can be easily proven in the standard way by analyzing difference quotients (see, for instance, \cite[Гл. V\!I, \S 2]{Mi}, \cite[Lm. 3.2]{BK1}), and thus we do not provide the proof.

We apply  Theorem~2.5 from \cite[Гл. V\!I, \S 2.1]{K} to obtain that adjoint operator for $\He$ corresponds to the adjoint form
$(\he)^*$ in the sense of the first representation theorem.
According to \cite[Гл. V\!I, \S 1.1]{K} and by idendities (\ref{3.3}) the adjoint form reads as
\begin{equation*}
(\he)^*(u,v)=\overline{\mathfrak{h}_\a^\e(u,v)}= \mathfrak{h}_{-\a}^\e(u,v)
\end{equation*}
that implies identity (\ref{1.11}). Identities (\ref{1.9}), (\ref{1.10}) now can be proven by straightforward calculations with employing the relations (\ref{1.3}).

It remains to prove inclusion (\ref{1.12}). The spectrum of $m$-sectorial operator is a subset of its numerical range \cite[Гл. V, \S 3.10]{K} and thus it is sufficient to prove the inclusion for the range of our operator. Immediately from estimates  (\ref{3.4}), (\ref{3.5}), (\ref{3.7}) and ellipticity condition (\ref{1.2}) the inequalities
\begin{equation}\label{3.11}
\begin{aligned}
&|\hei(u,u)|\leqslant 
 c_3\|\nabla u\|_{L_2(\Om^\e)} \|u\|_{L_2(\Om^\e)} + c_2\|u\|_{L_2(\Om^\e)}^2,
\\
&|\her(u,u)|\geqslant c_0\|\nabla u\|_{L_2(\Om^\e)}^2 - 2 c_1  \|\nabla u\|_{L_2(\Om^\e)} \|u\|_{L_2(\Om^\e)} - c_2 \|u\|_{L_2(\Om^\e)}^2
\end{aligned}
\end{equation}
follow. We let $\|u\|_{L_2(\Om^\e)}=1$, solve the second inequality w.r.t. $\|\nabla u\|_{L_2(\Om^\e)}$, and  substitute the obtained estimate for  $\|\nabla u\|_{L_2(\Om^\e)}$ into the first inequality for $\hei(u,u)$. Then we get
\begin{equation}\label{3.12}
\begin{aligned}
|\hei(u,u)|&\leqslant \frac{c_3}{c_0}\left(c_1+\sqrt{c_1^2+c_0(|\her(u,u)|+c_2)}\right)+c_2
\\
&\leqslant \frac{c_3}{\sqrt{c_0} }\sqrt{|\her(u,u)|}+\frac{(c_1+\sqrt{c_1^2+c_0 c_2})c_3}{c_0}+c_2\quad \text{при}\quad \|u\|_{L_2(\Om^\e)}=1.
\end{aligned}
\end{equation}
This inequality and the aforementioned inclusion of  spectrum into numerical range yield  (\ref{1.12}). The proof of Theorem~\ref{th2.1} is complete.

\section{Uniform resolvent convergence}

In the present section we study the behavior of the resolvent for operator as $\e\to+0$ and prove Theorem~\ref{th2.2}. Throughout the section by $C(\l)$ we denote inessential constants independent of $\e$, $x$, and $f$, but, generally speaking, they depend on $\l$.

By Theorem~\ref{th2.1} as $\l\in\mathds{C}\setminus\mathds{K}$ the resolvent $(\He-\l)^{-1}$ is well-defined. The next lemma is the key one in the proof of Theorem~\ref{th2.2}.
\begin{lemma}\label{lm4.1}
Let $\l\in\mathds{C}\setminus\mathds{K}$, $f\in L^\e_\bot$. Then the uniform in $\e$ and $\l$  estimate
\begin{equation*}
\|(\He-\l)^{-1}f\|_{\H^1(\Om^\e)}\leqslant \e C(\l)\|f\|_{L_2(\Om^\e)}.
\end{equation*}
\end{lemma}

\begin{proof}
Denote $v^\e:=(\He-\l)^{-1}f$. Then the results of the previous section and the belonging  $f\in L^\e_\bot$ follow that function $v^\e$ satisfies identity
\begin{equation}\label{4.1}
\he(v^\e,v^\e)-\l\|v^\e\|_{L_2(\Om^\e)}^2=(f,v^\e)_{L_2(\Om^\e)} =(f,v^\e_\bot)_{L_2(\Om^\e)},
\end{equation}
where $v^\e_\bot:=\mathcal{Q}^\e_\bot v^\e\in L^\e_\bot$. Since by the results of the previous section
$\he(v^\e,v^\e)/\|v^\e\|_{L_2(\Om^\e)}^2\in\mathds{K}$ (see (\ref{3.12})), by identity (\ref{4.1}) we obtain
\begin{equation}\label{4.2}
\|v^\e\|_{L_2(\Om^\e)}\leqslant \frac{\|f\|_{L_2(\Om^\e)} \|v^\e_\bot\|_{L_2(\Om^\e)}}{\dist(\mathds{K},\l)}.
\end{equation}

We take now the real part of identity (\ref{4.1}):
\begin{equation*}
\her(v^\e,v^\e)-\RE\l \|v^\e\|_{L_2(\Om^\e)}^2\leqslant \|f\|_{L_2(\Om^\e)}\|v^\e_\bot\|_{L_2(\Om^\e)}
\end{equation*}
and employ estimate (\ref{3.11}):
\begin{equation}\label{4.5}
\begin{aligned}
c_0\|\nabla v^\e\|_{L_2(\Om^\e)}^2&-2c_1 \|\nabla v^\e\|_{L_2(\Om^\e)} \|v^\e\|_{L_2(\Om^\e)}
\\
&- (c_2+\RE\l) \|v^\e\|_{L_2(\Om^\e)}^2\leqslant \|f\|_{L_2(\Om^\e)}\|v^\e_\bot\|_{L_2(\Om^\e)}.
\end{aligned}
\end{equation}
By (\ref{4.2}) we deduce
\begin{equation}
\|\nabla v^\e\|_{L_2(\Om^\e)}^2\leqslant C(\l) \|f\|_{L_2(\Om^\e)}\|v^\e_\bot\|_{L_2(\Om^\e)}.\label{4.3}
\end{equation}

Expanding function $v^\e$ into the standard Fourier series in variable $x_n$, we arrive at the inequality
\begin{equation}\label{4.4a}
\|v^\e_\bot\|_{L_2(\Om^\e)} \leqslant \frac{\e}{2\pi}\Big\|\frac{\p v^\e}{\p x_n}\Big\|_{L_2(\Om^\e)} \leqslant \frac{\e}{2\pi} \|\nabla v^\e\|_{L_2(\Om^\e)}.
\end{equation}
We substitute this estimate first into the left hand side of  (\ref{4.3}):
\begin{equation}\label{4.4}
\|v^\e_\bot\|_{L_2(\Om^\e)}\leqslant \e^2 C(\l) \|f\|_{L_2(\Om^\e)},
\end{equation}
and then into the right hand side:
\begin{equation}\label{4.5a}
\|\nabla v^\e\|_{L_2(\Om^\e)}\leqslant \e C(\l) \|f\|_{L_2(\Om^\e)}.
\end{equation}
It follows now from inequality (\ref{4.2}) that
\begin{equation}\label{4.6}
\|v^\e\|_{L_2(\Om^\e)}\leqslant  \e C(\l) \|f\|_{L_2(\Om^\e)}.
\end{equation}
Together with (\ref{4.5a}) it completes the proof.
\end{proof}

In the next lemma we prove self-adjointness of operator $\Ho$ and estimate its resolvent.

\begin{lemma}\label{lm4.2}
Operator $\Ho$ is self-adjoint. For each $\l\not\in\spec(\Ho)$ and each $F\in L_2(\mathds{R}^{d-1})$ the estimate
\begin{equation*}
\|(\Ho-\l)^{-1}F\|_{\H^2(\mathds{R}^{n-1})} \leqslant C(\l)\|F\|_{L_2(\mathds{R}^{n-1})}
\end{equation*}
holds true, where $C(\l)$ is a constant independent of $F$.
\end{lemma}

\begin{proof}
It follows from identities (\ref{1.3}), (\ref{1.13}) that functions $A_{ij}^0$, $A_0^0$ are real and the belongings $A_{ij}^0, A_j^0\in \Hinf^1(\mathds{R}^{n-1})$, $A_0^0\in L_\infty(\mathds{R}^{d-1})$ hold true.
The coefficient $A_{nn}(x',\xi,0)$ is positive and is uniformly separated from zero by ellipticity condition (\ref{1.2}) with $\z=\z_*:=(0,\ldots,0,1)$. Let us show that similar ellipticity condition holds also for coefficients  $A_{ij}^0$.

By $\mathrm{A}$ we denote the matrix with coefficients  $A_{ij}(x',\xi,0)$ and let $\z':=(\z_1,\ldots,\z_{n-1},0)$, $\z_j\in\mathds{R}$. Then by (\ref{1.13}) and (\ref{1.2}) we get
\begin{equation}\label{4.7}
\sum\limits_{i,j=1}^{n-1} A_{ij}^0\z_i \z_j= \int\limits_{-1/2}^{1/2} \frac{1}{A_{nn}(x',\xi,0)} \big((\mathrm{A}\z_*,z_*)_{\mathds{R}^n} (\mathrm{A}\z',\z')_{\mathds{R}^n} - (\mathrm{A}\z',\z_*)_{\mathds{R}^{n-1}}^2\big)\di \xi.
\end{equation}
Due to condition (\ref{1.3}) the form $(\mathrm{A}\cdot,\cdot)_{\mathds{R}^n}$ can be employed as an equivalent scalar product in $\mathds{R}^n$ and thus by Cauchy-Schwarz inequality we obtain
\begin{equation*}
(\mathrm{A}\z',\z_*)_{\mathds{R}^n}^2< (\mathrm{A}\z',\z')_{\mathds{R}^n}(\mathrm{A}\z_*,\z_*)_{\mathds{R}^n}.
\end{equation*}
This inequality is strict since vectors $\z'$ and $\z_*$ are noncollinear. Therefore,
\begin{equation*}
\min\limits_{\|\z'\|_{\mathds{R}^{n-1}}=1} \big((\mathrm{A}\z',\z')_{\mathds{R}^n}(\mathrm{A}\z_*,\z_*)_{\mathds{R}^n}
-(\mathrm{A}\z',\z_*)_{\mathds{R}^n}^2\big)\geqslant C>0.
\end{equation*}
By (\ref{4.7}) it yields the desired ellipticity condition.

The proof of self-adjointness for operator $\Ho$ now can be easily performed by analogy with the proof of  $m$-sectoriality for operator $\He$ in previous section. The analogue of form $\he$ is
\begin{align*}
\ho(u,v):=&\sum\limits_{i,j=1}^{n-1} \left(A_{ij}^0\frac{\p u}{\p x_j},\frac{\p u}{\p x_i}\right)_{L_2(\mathds{R}^{n-1})} + \sum\limits_{j=1}^{n-1} \left(A_j^0\frac{\p u}{\p x_j},v\right)_{L_2(\mathds{R}^{n-1})}
\\
&+\sum\limits_{j=1}^{n-1} \left(u,A_j^0\frac{\p v}{\p x_j}\right)_{L_2(\mathds{R}^{n-1})} + (A_0^0u,v)_{L_2(\mathds{R}^{n-1})}.
\end{align*}
The desired estimate for the resolvent can be proven on the basis of second fundamental inequality, cf. \cite[Гл. I\!I\!I, \S 8]{Ld}.
\end{proof}

We proceed to the proof of Theorem~\ref{th2.2}. Denote $f\in L_2(\Om^\e)$, ${\l\in \mathds{C}\setminus (\mathds{K}\cup \spec(\Ho))}$. We let
\begin{align*}
&
F_\e:=\e^{-1/2}\mathcal{Q}_\e f, && F_\bot^\e:=\mathcal{Q}_\bot^\e f, && u^\e:=(\He-\l)^{-1}f,  \\ &U^\e:=(\He-\l)^{-1}F^\e,
&& u^0:=(\Ho-\l)^{-1}F^\e.
\end{align*}
It is clear that
\begin{equation}\label{4.9}
\|F^\e\|_{L_2(\Om^\e)}^2+\|F_\bot^\e\|_{L_2(\Om^\e)}^2=\|f\|_{L_2(\Om^\e)}^2,
\end{equation}
and by Lemma~\ref{lm4.1}
\begin{equation}\label{4.10}
\|u^\e-U^\e\|_{\H^1(\Om^\e)}\leqslant \e C(\l)\|f\|_{L_2(\Om^\e)}.
\end{equation}
Hence, it is sufficient to estimate the norm of the difference $U^\e-u^0$. In order to do it, we introduce an additional corrector, which will play the key role. In fact, this corrector is the second term in the asymptotic expansion of function $U^\e$, if the latter is constructed on the basis of multiscale method. Namely, we  let
\begin{align*}
&w(x',\xi,\e):=-\sum\limits_{i,j=1}^{n-1} \frac{\p u^0}{\p x_j}(x') \int\limits_{0}^{\xi} \frac{A_{nj}(x',t,\e)}{A_{nn}(x',t,\e)}\di t - u^0(x') \int\limits_{0}^{\xi} \overline{A}_n(x',t,\e)\di t
\\
&\hphantom{w(x',\xi,\e):=}-\iu \a^\e(x') u^0(x') \int\limits_{0}^{\xi} \frac{\di t}{A_{nn}(x',t,\e)},
\\
&w^\e(x):=w\left(x',\frac{x_n}{\e},\e\right)\equiv (\mathcal{W}^\e u^0)\left(x',\frac{x_n}{\e},\e\right).
\end{align*}
It is clear that $w^\e\in\H^1(\Om^\e)$ for each $\e\in[0,\e_0]$. Function $w^\e$ is the above mentioned corrector and in what follows our aim is to estimate the norm of difference of functions $v^\e(x):=U^\e(x)-W^\e(x)$, $W^\e(x):=u^0(x')-\e w^\e(x)$.

We first observe that function $\om$ solves the equation
\begin{equation}\label{4.16}
A_{nn}\frac{\p \om}{\p \xi}+\sum\limits_{j=1}^{n-1} A_{nj} \frac{\p u^0}{\p x_j} + (\overline{A}_n+\iu \a) u^0=0,
\quad  (x',\xi,\e)\in\overline{\Pi}.
\end{equation}

It follows from the definition of functions $U^\e$ and the proofs of Theorem~\ref{th2.1} and of Lemma~\ref{lm4.1} that this function satisfies the integral identity
\begin{equation}\label{4.12}
\he(U^\e,v)-\l(U^\e,v)_{L_2(\Om^\e)}=(F^\e,v)_{L_2(\Om^\e)}
\end{equation}
for each $v\in\H^1(\Om^\e)$. And Lemma~\ref{lm4.2} together with the smoothness of the coefficients of operator $\Ho$ implies the relation
\begin{equation}\label{4.13}
F_\e=\left(-\sum\limits_{i,j=1}^{n-1}\frac{\p\hphantom{x}}{\p x_i} A_{ij}^0 \frac{\p\hphantom{x}}{\p x_j}  + \sum\limits_{j=1}^{n-1} \left( A_j^0 \frac{\p\hphantom{x}}{\p x_j} - \frac{\p\hphantom{x}}{\p x_j} \overline{A_j^0}\right)  + (A_0^0-\l) \right)u^0.
\end{equation}

We let  $v=v^\e$ in (\ref{4.12}) and employ identity $U^\e=v^\e+u^0+\e w^\e$. Then we get
\begin{equation}\label{4.17}
\he(v^\e,v^\e)-\l\|v^\e\|_{L_2(\Om^\e)}=(F^\e,v^\e)_{L_2(\Om^\e)} -\he(W^\e,v^\e)+\l(W^\e,v^\e)_{L_2(\Om^\e)}.
\end{equation}
The main idea of obtaining estimate for $v^\e$ is that first we transform the right hand side of the latter identity to a more convenient form and then we estimate it by a small quantity. It will imply the estimate for $v^\e$.

Let us transform the right hand side. Integrating by parts, we have
\begin{align}\label{4.8}
&
\he(u^0,v^\e)-\l(u^0,v^\e)_{L_2(\Om^\e)}=(g_1^\e+g_2^\e,v^\e)_{L_2(\Om^\e)}
 + \bb\left(\frac{\p u^0}{\p\nu^\e}+\iu\a^\e u^0,v^\e\right),
\\
&g_1^\e:=-\sum\limits_{i,j=1}^{n-1}\frac{\p\hphantom{x}}{\p x_i} A_{ij}^\e \frac{\p u^0}{\p x_j}  + \sum\limits_{j=1}^{n-1} \left( A_j^\e \frac{\p\hphantom{x}}{\p x_j} - \frac{\p\hphantom{x}}{\p x_j} \overline{A_j^\e}\right) u^0 + (A_0^\e-\l)u^0,\nonumber
\\
&g_2^\e:= -\sum\limits_{j=1}^{n-1}\frac{\p A_{nj}^\e}{\p x_n} \frac{\p u^0}{\p x_j}  - \frac{\p \overline{A_n^\e}}{\p x_n}u^0.
\nonumber
\end{align}
We integrate by parts in the same way,
\begin{equation}\label{4.11}
\begin{aligned}
&\e \sum\limits_{i=1}^{n} \left(A_{in}^\e\frac{\p\om^\e}{\p x_n}, \frac{\p v^\e}{\p x_i}\right)_{L_2(\Om^\e)}
=  -\e\sum\limits_{i=1}^{n} \left(\frac{\p\hphantom{x}}{\p x_i} A_{in}^\e \frac{\p w^\e}{\p x_n}, v^\e\right)_{L_2(\Om^\e)}
+
\e\bb\left(A_{nn}^\e\frac{\p \om^\e}{\p x_n},v^\e\right),
\\
&\iu\e\bb(\a^\e w^\e,v^\e)=\iu\e \left(\a^\e w^\e, \frac{\p v^\e}{\p x_n}\right)_{L_2(\Om^\e)} + \iu\e \left(\a^\e\frac{\p w^\e}{\p x_n}, v^\e\right)_{L_2(\Om^\e)}.
\end{aligned}
\end{equation}
By (\ref{4.16}) we get
\begin{align*}
&\e A_n^\e\frac{\p\om^\e}{\p x_n}
=-\frac{
A_n^\e}{A_{nn}^\e}
\left(\sum\limits_{j=1}^{n-1} A_{nj}^\e \frac{\p u^0}{\p x_j} + (\overline{A_n^\e}+\iu \a^\e) u^0\right),
\\
-&\e\sum\limits_{i=1}^{n} \frac{\p\hphantom{x}}{\p x_i} A_{in}^\e \frac{\p w^\e}{\p x_n}= \sum\limits_{i=1}^{n} \frac{\p\hphantom{x}}{\p x_i} \frac{A_{in}^\e}{A_{nn}^\e} \left( \sum\limits_{j=1}^{n-1} A_{nj}^\e \frac{\p u^0}{\p x_j} + (\overline{A_n^\e}+\iu \a^\e) u^0\right).
\end{align*}
The latter identities and (\ref{4.8}), (\ref{4.11}), (\ref{4.16}) imply
\begin{align}
&\he(W^\e,v^\e)-\l(W^\e,v^\e)=\e G_1^\e(w^\e,v^\e)+G_2^\e(u^0,v^\e)+G_3^\e(u^0,v^\e),
\label{4.15}
\\
&G_1^\e(w^\e,v^\e):=\sum\limits_{i=1}^{n-1}\sum\limits_{j=1}^{n} \left(A_{ij}^\e \frac{\p w^\e}{\p x_i},\frac{\p v^\e}{\p x_j}\right)_{L_2(\Om^\e)}+ \sum\limits_{j=1}^{n-1} \left( A_j^\e \frac{\p w^\e}{\p x_j}, v^\e\right)_{L_2(\Om^\e)} \nonumber
\\
&\hphantom{G_1^\e(w^\e,v^\e)=}+ \sum\limits_{j=1}^{n-1} \left( w^\e, A_j^\e \frac{\p v^\e}{\p x_j}\right)_{L_2(\Om^\e)} + ((A_0^\e-\l_0)w^\e,v^\e)_{L_2(\Om^\e)} 
+\iu\e \left(\a^\e w^\e,\frac{\p v^\e}{\p x_n}\right)_{L_2(\Om^\e)}, \nonumber
\\
&G_2^\e(u^0,v^\e):=\iu\sum\limits_{j=1}^{n-1} \left(u^0\frac{\p\hphantom{x}}{\p x_j} \left(\a^\e\frac{\p A_{jn}^\e}{A_{nn}^\e}\right),v^\e\right)_{L_2(\Om^\e)}\nonumber
\\
&G_3^\e(u^0,v^\e):=-\sum\limits_{i,j=1}^{n-1}\left(
\frac{\p\hphantom{x}}{\p x_i} B_{ij} \frac{\p u^0}{\p x_j} ,v^\e\right)_{L_2(\Om^\e)} + \sum\limits_{j=1}^{n-1} \left( B_j \frac{\p u^0}{\p x_j},v^\e\right)_{L_2(\Om^\e)} \nonumber
\\
&\hphantom{G_3^\e(u^0,v^\e)=}- \sum\limits_{j=1}^{n-1} \left( \frac{\p\hphantom{x}}{\p x_j} \overline{B_j}  u^0,v^\e\right)_{L_2(\Om^\e)} + \left((B_0-\l)u^0,v^\e\right)_{L_2(\Om^\e)}, \nonumber
\end{align}
where $B_{ij}^\e(x)=B_{ij}(x',x_n\e^{-1},\e)$,
$B_j^\e(x)=B_j(x',x_n\e^{-1},\e)$,
$B_0^\e(x)=B_0(x',x_n\e^{-1},\e)$,
${B_{ij}=B_{ij}(x',\xi,\e)}$, $B_j=B_j(x',\xi,\e)$, $B_0=B_0(x',\xi,\e)$,
\begin{equation}
\begin{aligned}
&B_{ij}:=A_{ij}-\frac{A_{in} A_{nj} }{A_{nn}},
\quad B_j:=A_j- \frac{A_n A_{nj}}{A_{nn}},
\\
&B_0:=A_0+ \frac{\a^2}{A_{nn}}  -\iu \a  \frac{A_n+\overline{A_n}}{A_{nn}}-\frac{|A_n|^2}{A_{nn}}.
\end{aligned}\label{4.22}
\end{equation}
Immediately from the definition of function $w^\e$ and Lemma~\ref{lm4.2} it follows the estimate
\begin{equation}\label{4.19}
|G^1_\e(w^\e,v^\e)|\leqslant \e C(\l)\|f\|_{L_2(\Om^\e)} \|v^\e\|_{\H^1(\Om^\e)}.
\end{equation}

We let $v^\e_\bot:=\mathcal{Q}^\e_\bot v^\e$. By conditions  (\ref{1.3}) the identities hold
\begin{equation*}
\int\limits_{-\frac{\e}{2}}^{\frac{\e}{2}} \frac{A_{jn}\left(x',\frac{x_n}{\e},\e\right)}{A_{nn}\left(x',\frac{x_n}{\e},\e\right)}\di x_n=0 \quad \text{for each} \quad x'\in \mathds{R}^{n-1},\quad \e\in[0,\e_0].
\end{equation*}
Since functions $\a^\e$ and $u^0$ are independent
of $\xi$, we have $u^0\frac{\displaystyle\p\hphantom{x}}{\displaystyle\p x_j}\a^\e\frac{\displaystyle A_{jn}^\e}{\displaystyle A_{nn}^\e}\in L^\e_\bot$. Then quantity $G_2^\e$ can be rewritten as
\begin{equation*}
G_2^\e(u^0,v^\e)= \iu\sum\limits_{j=1}^{n-1} \left(u^0\frac{\p\hphantom{x}}{\p x_j} \left(\a^\e\frac{\p A_{jn}^\e}{A_{nn}^\e}\right),v^\e_\bot\right)_{L_2(\Om^\e)}
\end{equation*}
that by Lemma~\ref{lm4.2} and inequality (\ref{4.4a}) yield the estimate
\begin{equation}\label{4.20}
|G_2^\e(u^0,v^\e)|\leqslant \e C(\l) \|f\|_{L_2(\Om^\e)} \|v^\e\|_{\H^1(\Om^\e)}.
\end{equation}
We let
\begin{align*}
G_\e^4(u^0,v^\e):=&-\sum\limits_{i,j=1}^{n-1}\left(
\frac{\p\hphantom{x}}{\p x_i} B_{ij}^0 \frac{\p u^0}{\p x_j} ,v^\e\right)_{L_2(\Om^\e)} + \sum\limits_{j=1}^{n-1} \left( B_j^0  \frac{\p u^0}{\p x_j},v^\e\right)_{L_2(\Om^\e)} \nonumber
\\
& - \sum\limits_{j=1}^{n-1} \left( \frac{\p\hphantom{x}}{\p x_j} \overline{B_j^0}   u^0,v^\e\right)_{L_2(\Om^\e)} + \left((B_0^0 -\l)u^0,v^\e\right)_{L_2(\Om^\e)}, \nonumber
\end{align*}
where
\begin{equation}\label{4.23}
B_{ij}^0(x,\e):=B_{ij}\left(x',\frac{x_n}{\e},0\right),\quad
B_j^0(x,\e):=B_j\left(x',\frac{x_n}{\e},0\right), \quad
B_0^0(x,\e):=B_0\left(x',\frac{x_n}{\e},0\right).
\end{equation}
Then the smoothness and the boundedness of functions $A_{ij}$, $A_j$, $A_0$ and of their derivatives and Lemma~\ref{lm4.2} yield
\begin{equation}\label{4.21}
|G^\e_4(u^0,v^\e)-G^\e_3(u^0,v^\e)|\leqslant \eta(\e) C(\l)\|f\|_{L_2(\Om^\e)} \|v^\e\|_{L_2(\Om^\e)}.
\end{equation}

It follows immediately from definition (\ref{1.13}) of functions $A_{ij}^0$, $A_j^0$, $A_0^0$ and definition (\ref{4.22}), (\ref{4.23}) of functions $B_{ij}^0$, $B_j^0$, $B_0^0$ that
\begin{equation*}
\int\limits_{-\frac{\e}{2}}^{\frac{\e}{2}} \big(A_\flat^0(x')-B_\flat^0(x,\e)\big)\di x_n=0,\quad \flat=ij,\quad \flat=j,\quad \flat=0.
\end{equation*}
These identities and Lemma~\ref{lm4.2} yield that the relation
\begin{equation*}
(F_\e,v^\e)_{L_2(\Om^\e)}-G_\e^4(u^0,v^\e)=(g^\e_3,v^\e)_{L_2(\Om^\e)}
\end{equation*}
is satisfied, where function $g^\e_3$ belongs to space $L^\e_\bot$ and obeys the estimate
\begin{equation*}
\|g^\e_3\|_{L_2(\Om^\e)}\leqslant C(\l)\|f\|_{L_2(\Om^\e)}.
\end{equation*}
By analogy with the deducing of (\ref{4.20}), one can check easily the inequality
\begin{equation*}
|(F_\e,v^\e)_{L_2(\Om^\e)}-G_\e^4(u^0,v^\e)|\leqslant  (\e+\eta(\e))  C(\l)\|f\|_{L_2(\Om^\e)} \|v^\e\|_{L_2(\Om^\e)}.
\end{equation*}
By this estimate and (\ref{4.17}), (\ref{4.15}), (\ref{4.19}), (\ref{4.20}), (\ref{4.21}) we conclude that the right hand side in (\ref{4.17}) is estimated by the quantity $(\e+\eta(\e)) C(\l)\|f\|_{L_2(\Om^\e)}\|v^\e|_{\H^1(\Om^\e)}$. Now it is sufficient to reproduce the arguments in the proof of Lemma~\ref{lm4.1} to obtain the required estimate for $v^\e$:
\begin{equation*}
\|v^\e\|_{\H^1(\Om^\e)}\leqslant (\e+\eta(\e))  C(\l)\|f\|_{L_2(\Om^\e)}.
\end{equation*}
We also observe one more obvious inequality implied directly from the definition of function $w^\e$:
\begin{equation*}
\|w^\e\|_{L_2(\Om^\e)}\leqslant \e C(\l)\|f\|_{L_2(\Om^\e)}.
\end{equation*}
The last two estimates and (\ref{4.10}) imply the statement of Theorem~\ref{th2.2}.

\section{Convergence of spectrum}

This section is devoted to the proof of Theorem~\ref{th2.3}.

Let $\l$ lie in a compact set in the complex plane and
$\mu\in\mathds{C}$ be a fixed number separated by the distance $1$ from the set $\mathds{K}\cup\spec(\Ho)$. Consider the equation
\begin{equation}\label{5.1}
(\He-\l)u=f,\quad f\in L_2(\Om^\e)
\end{equation}
and let us study its solvability. We rewrite it as
\begin{equation*}
(\He-\mu+\mu-\l)u=f
\end{equation*}
and apply operator $(\He-\mu)^{-1}$ which is well-defined by Theorem~\ref{th2.2}. Denoting ${\mathcal{T}^\e_1:=(\He-\mu)^{-1}-(\Ho-\mu)^{-1}\mathcal{Q}^\e\oplus0}$, we get
\begin{equation}\label{5.2}
u+(\mu-\l)(\Ho-\mu)^{-1}\mathcal{Q}^\e u\oplus 0 + (\mu-\l) \mathcal{T}^\e_1 u=f^\e_1,\quad f^\e_1:=(\He-\mu)^{-1}f.
\end{equation}
By Theorem~\ref{th2.2} the norm of operator $\mathcal{T}^\e_1: L_2(\Om^\e)\to L_2(\Om^\e)$ tends to zero as $\e\to+0$. This is why for all sufficiently small $\e$  the operator $\I+(\mu-\l)\mathcal{T}^\e_1$ is invertible and
\begin{align}
&u=\mathcal{T}^\e_2(\l) f^\e_1-(\mu-\l) \mathcal{T}^\e_2(\l) (\Ho-\mu)^{-1}\mathcal{Q}^\e u\oplus 0,\label{5.3}
\\
&\mathcal{T}^\e_2(\l):=(\I+(\mu-\l)\mathcal{Q}^\e\mathcal{T}^\e_1)^{-1}.\nonumber
\end{align}
This identity means that to solve equation (\ref{5.1}), it is sufficient to find function $\mathcal{Q}^\e u$. We also note that operator $\mathcal{T}^\e_2(\l)$ is holomorphic w.r.t. $\l$. The derivative of this operator w.r.t. $\l$ is as follows,
\begin{equation}\label{5.12}
\frac{\p \mathcal{T}^\e_2}{\p\l}(\l)= (\I+(\mu-\l)\mathcal{Q}^\e\mathcal{T}^\e_1)^{-1}
\mathcal{Q}^\e\mathcal{T}^\e_1
(\I+(\mu-\l)\mathcal{Q}^\e\mathcal{T}^\e_1)^{-1},
\end{equation}
and its norm tends to zero as $\e\to+0$.

We apply operator $\mathcal{Q}^\e$ to equation (\ref{5.2}) and substitute then formula (\ref{5.3}) into the definition of $\mathcal{T}^\e_1$ and employ then easily checked identities
\begin{equation*}
\I+(\mu-\l) (\Ho-\mu)^{-1}=(\Ho-\l)(\Ho-\mu)^{-1}, \quad
(\Ho-\mu)^{-1}\mathcal{Q}^\e=\mathcal{Q}^\e(\Ho-\mu)^{-1}\mathcal{Q}^\e.
\end{equation*}
We then obtain
\begin{align}
&(\Ho-\l+\mathcal{T}^\e_3(\l) )(\Ho-\mu)^{-1}\mathcal{Q}^\e u=f^\e_2,\label{5.4}
\\
&f^\e_2:=\big(\I-(\mu-\l) \mathcal{Q}^\e \mathcal{T}_\e^1(\l)\mathcal{T}_\e^2(\l)
\big)\mathcal{Q}^\e f_\e^1,\nonumber
\\
&\mathcal{T}^\e_3(\l):=-(\mu-\l)^2 \mathcal{Q}^\e \mathcal{T}^\e_1 \mathcal{T}^\e_2(\l)\mathcal{Q}^\e.\nonumber
\end{align}
We consider the obtained identity as an equation for $(\Ho-\mu)^{-1} \mathcal{Q}^\e u$. It is equivalent to the original equation  (\ref{5.1}), since once we find  $(\Ho-\mu)^{-1} \mathcal{Q}^\e u$, by formula (\ref{5.3}) we can recover solution to equation (\ref{5.1}).

We note that identifying spaces $L_2(\mathds{R}^{n-1})$ and $L^\e$, for each $v\in L_2(\mathds{R}^{n-1})$ we get the obvious identity $\|\mathcal{Q}^\e v\|_{L_2(\Om^\e)}=\e^{1/2}\|v\|_{L_2(\mathds{R}^{n-1})}$. By the definition of operator $\mathcal{T}^\e_3(\l): L_2(\mathds{R}^{n-1})\to L_2(\mathds{R}^{n-1})$ it yields that its norm tends to zero as $\e\to+0$. Moreover, this operator is holomorphic w.r.t.  $\l$ and by (\ref{5.12}) the norm of its derivative w.r.t.  $\l$ vanishes in the limit $\e\to+0$.

If now $\l$ is separated from the spectrum of operator $\Ho$ for all sufficiently small $\e$, equation (\ref{5.4}) is uniquely solvable
\begin{equation}\label{5.5}
(\Ho-\mu)^{-1}\mathcal{Q}^\e u=(\Ho-\l+\mathcal{T}_3^\e(\l))^{-1}f_2^\e.
\end{equation}
Therefore, the spectrum of operator $\He$ converges to that of operator $\Ho$ in the sense as it was stated in Theorem~\ref{th2.3}.

Let $\l_0$  be an isolated $m$-multiple eigenvalue of operator $\Ho$, $\phi_1$,~\ldots, $\phi_m$ be the associated eigenfunctions orthonormalized in $L_2(\mathds{R}^{n-1})$.  We let $f=0$ in (\ref{5.1}), i.e., we shall consider the eigenvalue equation for the perturbed operator. For its study we shall make use of the modified Birman-Schwinger principle from \cite{G}, \cite{MSb}.

For $\l$ close to $\l_0$ the representation \cite[Гл. V, \S 3.5]{K}
\begin{equation}\label{5.6}
(\Ho-\l)^{-1}=\sum\limits_{j=1}^{m} \frac{(\cdot,\phi_j)_{L_2(\mathds{R}^{n-1})}}{\l_0-\l}\phi_j + \mathcal{T}_4(\l)
\end{equation}
holds true, where operator $\mathcal{T}_4(\l)$ acts from  $L_2(\mathds{R}^{n-1})$ into the subset $\H^2(\mathds{R}^{n-1})$ comprising the functions orthogonal to  $\phi_1$, \ldots, $\phi_m$ in $L_2(\mathds{R}^{n-1})$. Moreover, operator $\mathcal{T}_4(\l)$ is holomorphic w.r.t. $\l$ in a sufficiently small neighborhood of $\l_0$. We denote
\begin{equation}\label{5.8}
U:=(\Ho-\l)^{-1}\mathcal{Q}^\e u
\end{equation}
and invert the operator $(\Ho-\l)$ in (\ref{5.4}) taking into consideration (\ref{5.6}):
\begin{equation*}
(\I+\mathcal{T}_4(\l)\mathcal{T}_3^\e(\l))U+\sum\limits_{j=1}^{m} \frac{(\mathcal{T}^\e_3(\l)U,\phi_j)_{L_2(\mathds{R}^{n-1})}}{\l_0-\l} \phi_j=0.
\end{equation*}
Since the operator $\mathcal{T}^\e_3(\l)$ is small and operator $\mathcal{T}_4(\l)$ is holomorphic, operator $(\I+\mathcal{T}_4(\l) \mathcal{T}^\e_3(\l))$ is invertible and
\begin{equation}\label{5.7}
U+
\sum\limits_{j=1}^{m} \frac{(\mathcal{T}^\e_3(\l)U,\phi_j)_{L_2(\mathds{R}^{n-1})}}{\l_0-\l} (\I+\mathcal{T}_4(\l)\mathcal{T}_3^\e(\l))^{-1}\phi_j=0.
\end{equation}
We let
\begin{equation}\label{5.9}
Z=(z_1,\ldots,z_m)^t,\quad z_j:=(\mathcal{T}^\e_3(\l)U,\phi_j)_{L_2(\mathds{R}^{n-1})}.
\end{equation}
As it follows from (\ref{5.7}), knowing quantities $z_j$, one can determine function $U$ and solve then equation (\ref{5.1}) with $f=0$ by means of (\ref{5.3}), (\ref{5.8}). In order to determine vector $z$, we apply operator $\mathcal{T}^\e_2(\l)$ to equation (\ref{5.7}) and calculate  then the scalar product with $\phi_i$ in $L_2(\mathds{R}^{n-1})$. Then we obtain the matrix equation
\begin{equation}\label{5.13}
\big((\l_0-\l)\mathrm{E}_m+\mathrm{B}^\e(\l)\big)Z=0,
\end{equation}
where $\mathrm{E}_m$ is the unit $m\times m$ matrix, $\mathrm{B}^\e(\l)$ is the matrix with the components
\begin{equation*}
A^\e_{ij}(\l):=\big(\mathcal{T}_3^\e(\l)
(\I+\mathcal{T}_4(\l)\mathcal{T}_3^\e(\l))^{-1}\phi_i,
\phi_j\big)_{L_2(R^{n-1})}.
\end{equation*}
The points at which the matrix $(\l_0-\l)\mathrm{E}_m+\mathrm{B}^\e(\l)$ is non-invertible are exactly the eigenvalues of operator $\He$. Indeed, if  $\l$ is one of such points, equation (\ref{5.13}) has finitely many linear independent solutions. By formulae (\ref{5.7}), (\ref{5.9}), (\ref{5.3}) with $f^\e_1=0$,  each such solution is associated with an eigenfunction of operator $\He$:
\begin{equation}\label{5.16}
u=\mathcal{T}^\e_2(\l)U,
\quad
U=\sum\limits_{j=1}^{m}  z_j (\I+\mathcal{T}_4(\l)\mathcal{T}_3^\e(\l))^{-1}\phi_j.
\end{equation}
The factors  $1/(\l-\l_0)$ in the formula for $U$ and  $(\mu-\l)$ in that for $u$  can be omitted since the eigenfunction is determined up to a multiplicative constant. It is also easy to make sure that linearly independent vectors $z$ are associated with linearly independent eigenfunctions of operator $\He$. Thus, the multiplicity of an eigenvalue $\l$ of operator $\He$ coincides with the number of linearly independent solutions to equation (\ref{5.13}).

The properties of operator $\mathcal{T}^\e_3(\l)$ and $\mathcal{T}_4(\l)$ imply that the entries of matrix $\mathrm{B}^\e(\l)$ are holomorphic w.r.t. $\l$. Moreover, these entries and their derivatives w.r.t. $\l$ tends to zero as
$\e\to+0$ uniformly in $\l$ in a small neighborhood of point $\l_0$. We denote
\begin{equation*}
R^\e(\l):= \det(\l-\l_0-\mathrm{B}^\e(\l)).
\end{equation*}

\begin{lemma}\label{lm5.1}
The function $\l\mapsto R^\e(\l)$ has exactly $m$ zeroes (counting the orders) converging to $\l_0$ as $\e\to+0$.
\end{lemma}

\begin{proof}
It is clear that
\begin{equation*}
R^\e(\l)=(\l-\l_0)^m+R^\e_1(\l),
\end{equation*}
where function $R^\e_1(\l)$ is holomorphic w.r.t. $\l$ in a small neighborhood of point $\l_0$ and tends to zero as $\e\to+0$ uniformly in $\l$. Employing this representation and applying Rouch\'e theorem, we complete the proof.
\end{proof}

Let $\l^\e$ be a zero of function $R^\e(\l)$ of order $k(\e)$ described in Lemma~\ref{lm5.1}. Equation (\ref{5.13}) has
$q(\e)$  linearly independent solutions associated with this zero and $q(\e)$ is also the multiplicity of $\l^\e$ regarded as an eigenvalues of operator l$\He$.
Let us prove that this multiplicity of  $\l^\e$ coincides with its order once we regard it as a zero of function $R^\e(\l)$.

\begin{lemma}\label{lm5.2}
For all sufficiently small $\e$ and all zeroes of function  $R^\e(\l)$ the identity $p(\e)=q(\e)$ holds true.
\end{lemma}

\begin{proof}
Let $Z_1$, \ldots, $Z_q$ be linearly independent solutions to equation (\ref{5.13}) associated with $\l=\l^\e$. Without loss of generality we choose vectors $Z_j$ orthonormalized in $\mathds{C}^m$. Since $q\leqslant m$, we complement these vectors by vectors $Z_j$, $j=q+1,\ldots,m$, in such a way that the obtained systems forms an orthonormalized basis in  $\mathds{C}^m$. By $\mathrm{S}$ we denote the matrix with the columns $Z_j$, $j=1,\ldots,m$. Since vectors $Z_j$ are orthonormalized, matrix $\mathrm{S}$  is non-degenerate and orthogonal $\mathrm{S}^{-1}=\mathrm{S}^*$.

By the aforementioned properties of matrix
$\mathrm{B}^\e(\l)$ and Hadamard lemma the representation
\begin{equation*}
\mathrm{B}^\e(\l)-\mathrm{B}^\e(\l^\e)=(\l-\l^\e) \mathrm{B}_1^\e(\l)
\end{equation*}
is valid, where matrix $\mathrm{B}_1^\e(\l)$ is holomorphic w.r.t. $\l$ in a small neighborhood of point $\l_0$, and its elements tends to zero as $\e\to+0$ uniformly in  $\l$.
Employing this representation, by straightforward  calculations
we check that
\begin{align}
&
\begin{aligned}
&\mathrm{S}^{-1}\big((\l-\l_0)\mathrm{E}_m-\mathrm{B}^\e(\l)\big)\mathrm{S}
= (\l-\l^\e)\big(\mathrm{E}_m+\mathrm{S}^{-1}\mathrm{B}^\e_1(\l)\mathrm{S}\big)\\
&+
\mathrm{S}^{-1}\big((\l^\e-\l_0)\mathrm{E}_m -\mathrm{B}^\e(\l^\e)\big)\mathrm{S}
 = \big(\mathrm{E}_m+\mathrm{S}^{-1}\mathrm{B}^\e_1(\l)\mathrm{S}\big) \big((\l-\l^\e)\mathrm{E}_m+
\mathrm{B}^\e_2\big),
\end{aligned}\label{5.14}
\\
&\mathrm{B}^\e_2(\l):=\mathrm{S}^{-1}\big(\mathrm{E}_m+ \mathrm{B}^\e_1(\l)\big)^{-1} \big((\l^\e-\l_0)\mathrm{E}_m -\mathrm{B}^\e(\l^\e)\big)\mathrm{S}\nonumber.
\end{align}
By the definition of matrix $\mathrm{S}$ matrix $\mathrm{B}^\e_2$ has a block structure
\begin{equation*}
\mathrm{B}^\e_2(\l)=
\begin{pmatrix}
0 & \mathrm{B}^\e_3(\l)
\\
0 & \mathrm{B}^\e_4(\l)
\end{pmatrix},
\end{equation*}
where the zero block in the upper left corner is of the size $q\times m$, the zero block in the left lower corner
has the size $(m-q)\times m$, while the blocks $\mathrm{B}^\e_3$ and $\mathrm{B}^\e_4$ are respectively of the size $q\times (m-q)$ and $(m-q)\times (m-q)$. By the properties of matrix $\mathrm{B}^\e_1$ it implies
\begin{gather}
R^\e(\l)=(\l-\l^\e)^q R^\e_1(\l), \label{5.15}
\\
R_\e^1(\l):=\det{}^{-1}(\mathrm{E}_m +\mathrm{S}^{-1}\mathrm{B}^\e_1(\l)\mathrm{S})\, \det\big((\l-\l^\e)\mathrm{E}_{m-q} -\mathrm{B}^\e_4(\l)\big).\nonumber
\end{gather}
Identity (\ref{5.14}) with $\l=\l^\e$ and the assumption
\begin{equation*}
\mathrm{rank}\, \big((\l^\e-\l_0)\mathrm{E}_m-\mathrm{B}^\e(\l^\e)\big)=m-q
\end{equation*}
yields that $\mathrm{rank}\,\mathrm{A}_4^\e(\l^\e)=m-q$, and hence $R_\e^1(\l^\e)\not=0$. Together with (\ref{5.15}) it completes the proof.
\end{proof}

\begin{remark}\label{rm5.1}
We note that similar lemma was proven  in \cite[Lm. 6.3]{MPAG07} within the framework on the basis of the modified Birman-Schwinger approach. At the same time, in the mentioned work the self-adjointness of both the perturbed and limiting operators was employed essentially. In the present work we succeeded to get rid of this assumption for the perturbed operator.
\end{remark}

Lemmata~\ref{lm5.1},~\ref{lm5.2} imply the second part of Theorem~\ref{th2.3} on convergence of eigenvalues. It remains to prove the reality for the eigenvalues of perturbed operator converging to $\l_0$.

Let $\l^\e$ be one of such eigenvalues, and $\psi^\e$ is the associated eigenfunction. Then $\psi^\e$ satisfies representation (\ref{5.16})  with $u$  replaced by $\psi^\e$. Normalizing nontrivial solution $Z$ to equation (\ref{5.13}) as follows
\begin{equation}\label{5.17}
\|Z\|_{\mathds{C}^m}=\e^{-1}.
\end{equation}
Then it follows from (\ref{5.16}), the definition of operators $\mathcal{T}^\e_2$, $\mathcal{T}^\e_3$, $\mathcal{T}_4$ and Theorem~\ref{th2.2} that
\begin{equation}\label{5.18}
\psi^\e=\sum\limits_{j=1}^{m} z_j\phi_j+\Odr(\e+\eta)\quad\text{in the norm of}\quad L_2(\Om^\e).
\end{equation}
Employing identities (\ref{1.9}), (\ref{1.11}) and eigenvalue equation for $\psi^\e$, by straightforward calculations we check that
\begin{equation}\label{5.19}
\begin{aligned}
0=&\big((\He-\l^\e)\psi^\e,\mathcal{P}\psi^\e\big)_{L_2(\Om^\e)}= \big(\psi^\e,(\mathcal{H}^\e_{-\a}-\overline{\l^\e})\mathcal{P}\psi^\e
\big)_{L_2(\Om^\e)}
\\
=& \big(\psi^\e,\mathcal{P}(\He-\overline{\l^\e})\psi^\e
\big)_{L_2(\Om^\e)}=(\overline{\l^\e}-\l^\e) (\psi^\e,\mathcal{P}\psi^\e)_{L_2(\Om^\e)}.
\end{aligned}
\end{equation}
It follows from the definition of operator $\mathcal{P}$ and (\ref{5.18}) that
\begin{equation}\label{5.21}
\mathcal{P}\psi^\e=\sum\limits_{j=1}^{m} z_j\phi_j+\Odr(\e+\eta)\quad\text{in the norm of}\quad L_2(\Om^\e),
\end{equation}
and thus by (\ref{5.17}), (\ref{5.18}) and the orthonormality of $\phi_j$ in $L_2(\mathds{R}^{n-1})$
\begin{equation}\label{5.20}
(\psi^\e,\mathcal{P}\psi^\e)_{L_2(\Om^\e)}=1+\Odr(\e+\eta).
\end{equation}
It yields that identity (\ref{5.19}) is possible only for real $\l^\e$. The proof of Theorem~\ref{th2.3} is complete.

\section{Asymptotic expansions: formal construction}

In the present section we provide the first part of the proof for Theorem~\ref{th2.4} which is the formal construction of asymptotic expansions for the eigenvalues and eigenfunctions of the perturbed operator. The second part of the proof, a rigorous justification and estimates for the error terms, will be given in the next section.

Let $\l^0$ be an isolated $m$-multiple eigenvalue of operator $\Ho$, and $\phi_k=\phi_k(x')$, $k=1,\ldots,m$, are the associated real-valued eigenfunctions orthonormalized in  $L_2(\mathds{R}^{n-1})$. In accordance with Theorem~\ref{th2.3}, there exist exactly $m$ eigenvalues $\l^\e_k$, $k=1,\ldots,m$, of perturbed operator converging to $\l^0$ as $\e\to+0$. We construct the asymptotics for these eigenvalues as
\begin{equation}\label{6.1}
\l_k^\e=\l^0+\sum\limits_{p=1}^{\infty} \e^p \L_k^{(p)},\quad k=1,\ldots,m,
\end{equation}
and the asymptotics of the associated eigenfunctions are constructed as
\begin{equation}\label{6.2}
\psi_k^\e(x)=\phi_k(x')+\sum\limits_{p=1}^{\infty} \e^p \phi_k^{(p)}(x',\xi),\quad k=1,\ldots,m,
\end{equation}
where $\xi=x_n\e^{-1}$ is a rescaled variable, $\L_k^{(p)}$ and $\psi_k^{(p)}$ are some numbers and functions  and to  determine them is the main aim of the formal construction. We construct the asymptotics by the multiscale method \cite{MS}.

In what follows it is convenient to regard the eigenfunctions of the perturbed operator as generalized solutions to the boundary value problem
\begin{equation}\label{6.3}
\begin{gathered}
\left( -\sum\limits_{i,j=1}^{n} \frac{\p\hphantom{x}}{\p x_i} A_{ij}^\e \frac{\p\hphantom{x}}{\p x_j} + \sum\limits_{j=1}^{n} \left(A_j^\e\frac{\p\hphantom{x}}{\p x_j}-\frac{\p\hphantom{x}}{\p x_j} \overline{A_j^\e}\right)+ A_0^\e
\right)\psi_k^\e=\l_k^\e\psi_k^\e\quad\text{in}\quad \Om^\e,
\\
\left(\frac{\p\hphantom{\nu}}{\p\nu^\e}+\iu\a^\e\right)\psi_k^\e=0\quad \text{on}\quad\p\Om^\e.
\end{gathered}
\end{equation}
We substitute series (\ref{6.1}), (\ref{6.2}) into this boundary value problem, take into consideration the dependence of functions $\phi_k^{(p)}$ on variable $\xi$, and collect then the coefficients at the like powers of $\e$. Then we obtain the recurrent system of boundary value problems
\begin{equation}\label{6.4}
\begin{aligned}
-&\frac{\p\hphantom{.}}{\p\xi}A_{nn} \frac{\p\phi_k^{(p)}}{\p\xi}+\frac{\p\hphantom{\nu}}{\p\nu^*}\frac{\p \phi_k^{(p-1)}}{\p\xi}-\frac{\p \hphantom{.} }{\p\xi}\frac{\p\phi_k^{(p-1)}}{\p\nu}
\\
&\hphantom{\frac{\p\hphantom{.}}{\p\xi}A_{nn} \frac{\p\phi_k^{(p)}}{\p\xi}}+ \mathcal{T}_5 \phi_k^{(p-2)}=\l^0\phi_k^{(p-2)}+\sum\limits_{q=1}^{p-2} \L_k^{(q)} \phi_k^{(p-q-2)}\quad\text{в}\quad\Om,
\\
& A_{nn}\frac{\p\phi_k^{(p)}}{\p\xi}+\frac{\p\phi_k^{(p-1)}}{\p\nu}=0 \quad\text{на}\quad \p\Om,\quad p\geqslant 1,
\end{aligned}
\end{equation}
where we have denoted
\begin{align*}
& \frac{\p\hphantom{\nu}}{\p\nu}:=\sum\limits_{j=1}^{n-1} A_{nj} \frac{\p\hphantom{x}}{\p x_j} + \overline{A}_n+\iu\a,
\quad \frac{\p\hphantom{\nu}}{\p\nu^*}:=-\sum\limits_{j=1}^{n-1} A_{nj} \frac{\p\hphantom{x}}{\p x_j} + A_n+\iu\a,
\\
&\mathcal{T}_5:=-\sum\limits_{i,j=1}^{n-1} \frac{\p\hphantom{x}}{\p x_i} A_{ij} \frac{\p\hphantom{x}}{\p x_j} + \sum\limits_{j=1}^{n-1} \left(A_j \frac{\p\hphantom{x}}{\p x_j}-\frac{\p\hphantom{x}}{\p x_j} \overline{A_j}\right)+ A_0,
\\
&\phi_k^{(0)}:=\phi_k,\quad \phi_k^{(-1)}:=0.
\end{align*}
In order to solve problem (\ref{6.4}), we shall make use of the following auxiliary lemma.

\begin{lemma}\label{lm6.1}
Let $F=F(x',\xi)$ be a function such that ${F(x',\cdot)\in L_2(-1/2,1/2)}$ for each $x'\in \mathds{R}^{n-1}$, $g_\pm=g_\pm(x')$ be some functions. The boundary value problem
\begin{equation}\label{6.5}
-\frac{\p\hphantom{.}}{\p\xi}A_{nn} \frac{\p\phi}{\p\xi}+F=0\quad\text{в}\quad\Om,\quad
A_{nn}\frac{\p\phi_k^{(p)}}{\p\xi}+g_\pm=0 \quad\text{при}\quad \xi=\pm1/2,
\end{equation}
is solvable if and only if
\begin{equation}\label{6.6}
\int\limits_{-\frac{1}{2}}^{\frac{1}{2}} F(x',\xi)\di \xi=g_-(x')-g_+(x')\quad \text{для всех}\quad x'\in \mathds{R}^{n-1}.
\end{equation}
There exists the unique solution to problem (\ref{6.5}) obeying condition
\begin{equation}\label{6.7}
\int\limits_{-\frac{1}{2}}^{\frac{1}{2}} \phi(x',\xi)\di \xi =0\quad \text{для всех}\quad x'\in \mathds{R}^{n-1}.
\end{equation}
This solution is given by the formula
\begin{equation}\label{6.8}
\begin{aligned}
\phi(x',\xi)=&-g_-(x') \left(\int\limits_{-\frac{1}{2}}^{\xi}\frac{\di t}{A_{nn}(x',t)}+ \int\limits_{-\frac{1}{2}}^{\frac{1}{2}} \frac{(t-\frac{1}{2})\di t}{A_{nn}(x',t)}\right)
\\
& + \left( \int\limits_{-\frac{1}{2}}^{\xi} \frac{\di t}{A_{nn}(x',t)} \int\limits_{-\frac{1}{2}}^{t} F(x',s)\di s+ \int\limits_{-\frac{1}{2}}^{\frac{1}{2}} \di t \frac{t-\frac{1}{2}}{A_{nn}(x',t)}\int\limits_{-\frac{1}{2}}^{t} F(x',s)\di s
\right).
\end{aligned}
\end{equation}
The general solution is the sum of the latter and an arbitrary function depending on $x'$ only.
\end{lemma}

\noindent The statement of this lemma can be checked by straightforward calculations.

We proceed to solving problems  (\ref{6.4}). We first consider independently these problems for $p=1,2,3$, and then we construct the solutions for arbitrary $p$. We have to consider separately the cases $p=1,2,3$,  since to construct the solution for arbitrary $p$, one has to employ constructions for the cases $p=1,2,3$.

As $p=1$, problem (\ref{6.4}) casts into the form
\begin{equation*}
-\frac{\p\hphantom{.}}{\p\xi}A_{nn} \frac{\p\phi_k^{(1)}}{\p\xi} -\frac{\p \hphantom{.} }{\p\xi}\frac{\p\phi_k }{\p\nu}=0 \quad\text{in}\quad\Om,
\quad
A_{nn}\frac{\p\phi_k^{(1)}}{\p\xi}+\frac{\p\phi_k}{\p\nu}=0 \quad\text{on}\quad \p\Om.
\end{equation*}
It implies
\begin{align}
& A_{nn} \frac{\p\phi_k^{(1)}}{\p\xi}+\frac{\p\phi_k }{\p\nu}=0,
\label{6.9}
\\
&\phi_k^{(1)}=\hat{\phi}_k^{(1)}+\Phi_k^{(1)},\quad \hat{\phi}_k^{(1)}:= \mathcal{T}_6\phi_k,\label{6.10}
\end{align}
where $\Phi_k^{(1)}=\Phi_k^{(1)}(x')$ is a function which will be determined below,
\begin{align}
&(\mathcal{T}_6\phi)(x',\xi):=\sum\limits_{j=1}^{n-1} G_j(x',\xi) \frac{\p\phi}{\p x_j}(x')+G_0(x',\xi)\phi(x'),\label{6.11}
\\
&G_j(x',\xi):=-\int\limits_{-\frac{1}{2}}^{\xi} \frac{A_{nj}(x',t)\di t}{A_{nn}(x',t)} - \int\limits_{-\frac{1}{2}}^{\frac{1}{2}} \frac{t A_{nj}(x',t)}{A_{nn}(x',t)}\di t,\nonumber
\\
&G_0(x',\xi):=-\int\limits_{-\frac{1}{2}}^{\xi} \frac{\overline{A}_{n}(x',t)+\iu\a(x')}{A_{nn}(x',t)}\di t - \int\limits_{-\frac{1}{2}}^{\frac{1}{2}} \left(t-\frac{1}{2}\right) \frac{\overline{A}_n(x',t)+\iu\a(x')}{A_{nn}(x',t)}\di t.\nonumber
\end{align}
In view of identities (\ref{1.3}) it is easy to check that function $\phi_k^{(1)}$ obeys condition (\ref{6.7}).

We write down problem (\ref{6.4}) for $p=2$:
\begin{equation}\label{6.12}
\begin{aligned}
-&\frac{\p\hphantom{.}}{\p\xi}A_{nn} \frac{\p\phi_k^{(2)}}{\p\xi}+\frac{\p\hphantom{\nu}}{\p\nu^*}\frac{\p \phi_k^{(1)}}{\p\xi}-\frac{\p \hphantom{.} }{\p\xi}\frac{\p\phi_k^{(1)}}{\p\nu}
+ \mathcal{T}_5 \phi_k=\l^0\phi_k \quad\text{in}\quad\Om,
\\
& A_{nn}\frac{\p\phi_k^{(2)}}{\p\xi}+\frac{\p\phi_k^{(1)}}{\p\nu}=0 \quad\text{on}\quad \p\Om.
\end{aligned}
\end{equation}
We write the solvability condition (\ref{6.6}):
\begin{equation*}
\int\limits_{-\frac{1}{2}}^{\frac{1}{2}} \frac{\p\hphantom{\nu}}{\p\nu^*}\frac{\p \phi_k^{(1)}}{\p\xi} \di \xi - \int\limits_{-\frac{1}{2}}^{\frac{1}{2}} \frac{\p \hphantom{.} }{\p\xi}\frac{\p\phi_k^{(1)}}{\p\nu} \di\xi + \int\limits_{-\frac{1}{2}}^{\frac{1}{2}} (\mathcal{T}_5 \phi_k-\l^0\phi_k)\di\xi=-\frac{\p\phi_k^{(1)}}{\p\nu} \bigg|^{\xi=\frac{1}{2}}_{\xi=-\frac{1}{2}},
\end{equation*}
which yields
\begin{equation*}
\int\limits_{-\frac{1}{2}}^{\frac{1}{2}}
\left(\frac{\p\hphantom{\nu}}{\p\nu^*}\frac{\p \hphantom{\xi}}{\p\xi}\mathcal{T}_6 +\mathcal{T}_5 -\l^0\right)\phi_k\di\xi=0.
\end{equation*}
Substituting here the expressions for  $\frac{\p\hphantom{\nu}}{\p\nu^*}$ and identities (\ref{6.9}), (\ref{6.10}), we obtain equations for eigenfunctions $\phi_k$:
\begin{equation*}
\Ho\phi_k=\l^0\phi_k,
\end{equation*}
which holds by the definition of eigenfunctions $\phi_k$ and eigenvalue $\l^0$.  Returning back to problem (\ref{6.12}), we substitute there formulae (\ref{6.9}), (\ref{6.10}), take into consideration the identity
\begin{equation}\label{6.17}
\frac{\p\phi^{(1)}_k}{\p\xi}=\frac{\p\hat{\phi}^{(k)}_1}{\p\xi}
\end{equation}
and write down then the solution  by Lemma~\ref{lm6.1}. As a result we get
\begin{align}\label{6.13}
&\phi_k^{(2)}(x',\xi)=\check{\phi}_k^{(2)}(x',\xi)+\hat{\phi}_k^{(2)}(x',\xi)+ \Phi_k^{(2)}(x'),
\\
&\hat{\phi}_k^{(2)}=\mathcal{T}_6\Phi_k^{(1)}, \quad \check{\phi}_k^{(2)}=\mathcal{T}_7\Phi_k^{(1)},\label{6.19}
\end{align}
where $\mathcal{T}_7 \phi$ is the function defined by formula (\ref{6.8}) with
\begin{equation*}
F=\left(\frac{\p\hphantom{\nu}}{\p\nu^*}\frac{\p\hphantom{\xi} }{\p\xi}\mathcal{T}_6-\frac{\p \hphantom{.} }{\p\xi}\frac{\p\hphantom{\nu}}{\p\nu}\mathcal{T}_6
+ \mathcal{T}_5 -\l^0\right)\phi,\quad g_-=\frac{\p\hphantom{\nu}}{\p\nu}\mathcal{T}_6\phi\bigg|_{\xi=-\frac{1}{2}}.
\end{equation*}

We proceed to the case $p=3$. Here problem (\ref{6.4}) casts into the form
\begin{equation}\label{6.14}
\begin{aligned}
-&\frac{\p\hphantom{.}}{\p\xi}A_{nn} \frac{\p\phi_k^{(3)}}{\p\xi}+\frac{\p\hphantom{\nu}}{\p\nu^*}\frac{\p \phi_k^{(2)}}{\p\xi}-\frac{\p \hphantom{.} }{\p\xi}\frac{\p\phi_k^{(2)}}{\p\nu}
+
 \mathcal{T}_5 \phi_k^{(1)}=\l^0\phi_k^{(1)}+  \L_k^{(1)} \phi_k \quad\text{in}\quad\Om,
\\
& A_{nn}\frac{\p\phi_k^{(3)}}{\p\xi}+\frac{\p\phi_k^{(2)}}{\p\nu}=0 \quad\text{on}\quad \p\Om.
\end{aligned}
\end{equation}
We write down solvability condition (\ref{6.6}) for this problem and bear in mind formulae (\ref{6.10}), (\ref{6.13}), (\ref{6.13}) and identity (\ref{6.7}) for $\hat{\phi}_k^{(1)}$, $\hat{\phi}_k^{(2)}$, $\check{\phi}_k^{(2)}$. We get
\begin{gather}
(\Ho-\l^0)\Phi_k^{(1)}=h_k^{(1)}+\l^{(1)}_k\phi_k,\label{6.15}
\\
h_k^{(1)}:=-\int\limits_{-\frac{1}{2}}^{\frac{1}{2}}
\left(
\frac{\p\hphantom{\nu}}{\p\nu^*}\frac{\p \check{\phi}_k^{(2)}}{\p\xi} + \mathcal{T}_5 \hat{\phi}_k^{(1)}
\right)
\di\xi=-\int\limits_{-\frac{1}{2}}^{\frac{1}{2}}
\left(
\frac{\p\hphantom{\nu}}{\p\nu^*}\frac{\p \hphantom{\xi}}{\p\xi} \mathcal{T}_7 + \mathcal{T}_5 \mathcal{T}_6
\right)\phi_k\di\xi.\nonumber
\end{gather}
Since $\l^0$ is an $m$-multiple eigenvalue of operator $\Ho$ and the latter is self-adjoint, the obtained equation is solvable if and only if the right hand side is orthogonal to all $\phi_s$, $s=-1,\ldots,m$, in $L_2(\mathds{R}^{n-1})$:
\begin{equation}\label{6.16}
(h_k^{(1)},\phi_s)_{L_2(\mathds{R}^{n-1})}+\l_k^{(1)}\d_{ks}=0,\quad k,s=1,\ldots,m,
\end{equation}
where $\d_{ks}$ is the Kronecker delta. Let us show that numbers $\L_k^{(1)}$ and functions $\phi_k$ can be chosen so that these identities are satisfied. We first prove that the matrix composed by the numbers $-(h_k^{(1)},\phi_s)_{L_2(\mathds{R}^{n-1})}$ is Hermitian. We indicate this matrix by $\mathrm{L}$.

The definition implies immediately that
\begin{equation}\label{6.20}
-(h_k^{(1)},\phi_s)_{L_2(\mathds{R}^{n-1})}= \left(\frac{\p\hphantom{\nu}}{\p\nu^*}\frac{\p\check{\phi}_k^{(2)}}{\p\xi}, \phi_s\right)_{L_2(\Om)}+ (\mathcal{T}_5\hat{\phi}_k^{(1)},\phi_s)_{L_2(\Om)}.
\end{equation}
Integrating by parts and employing  (\ref{6.9}), (\ref{6.17}), (\ref{6.19}), we have
\begin{align*} &\left(\frac{\p\hphantom{\nu}}{\p\nu^*}\frac{\p\check{\phi}_k^{(2)}}{\p\xi}, \phi_s\right)_{L_2(\Om)}= \left(\frac{\p\check{\phi}_k^{(2)}}{\p\xi}, \frac{\p \phi_s }{\p\nu^*}\right)_{L_2(\Om)}=-\left(\frac{\p\check{\phi}_k^{(2)}}{\p\xi}, A_{nn}\frac{\p \hat{\phi}_s^{(1)} }{\p\xi}\right)_{L_2(\Om)}
\\
&=-\int\limits_{\mathds{R}^{n-1}} A_{nn}\frac{\p\check{\phi}_k^{(2)}}{\p\xi} \overline{\hat{\phi}_s^{(1)}}\bigg|_{\xi=-\frac{1}{2}}^{\xi=\frac{1}{2}}\di x' + \left( \frac{\p}{\p\xi} A_{nn} \frac{\p\check{\phi}_k^{(2)}}{\p\xi}, \hat{\phi}_s^{(1)}\right)_{L_2(\Om)}
\\
&=-\int\limits_{\mathds{R}^{n-1}} A_{nn}\frac{\p\check{\phi}_k^{(2)}}{\p\xi} \overline{\hat{\phi}_s^{(1)}}\bigg|_{\xi=-\frac{1}{2}}^{\xi=\frac{1}{2}}\di x' - \left(\frac{\p \hphantom{.} }{\p\xi}\frac{\p\hat{\phi}_k^{(1)}}{\p\nu}-\frac{\p\hphantom{\nu}}{\p\nu^*}\frac{\p \hat{\phi}_k^{(1)}}{\p\xi}-\mathcal{T}_5\phi_k,  \hat{\phi}_s^{(1)}\right)_{L_2(\Om)}
\\
&= \left(\frac{\p\hat{\phi}_k^{(1)}}{\p\nu},
\frac{\p \hat{\phi}_s^{(1)}}{\p\xi}\right)_{L_2(\Om)}+
\left(\frac{\p \hat{\phi}_k^{(1)}}{\p\xi}, \frac{\p\hat{\phi}_s^{(1)}}{\p\nu}\right)_{L_2(\Om)}
+\left(\mathcal{T}_5\phi_k, \hat{\phi}_s^{(1)}\right)_{L_2(\Om)}.
\end{align*}
These identities and (\ref{6.20}) prove that matrix $\mathrm{L}$ is Hermitian. By the theorem on simultaneous diagonalization of two quadratic forms we conclude that keeping eigenfunctions
$\phi_k$ orthonormalized in  $L_2(\mathds{R}^{n-1})$, we can choose them so that $\mathrm{L}$ is diagonal. In this case identities (\ref{6.16}) obviously hold true, once we let $\L_k^{(1)}$ equal to the eigenvalues of matrix $\mathrm{L}$. In what follows the quantities $\L_k^{(1)}$ and  eigenfunctions $\phi_k$ are assumed to chosen exactly in this way.

Since solvability conditions (\ref{6.16}) are satisfied, equation (\ref{6.15}) has the unique solution orthogonal to all eigenfunctions $\phi_s$, $s=1,\ldots,m$, in $L_2(\mathds{R}^{n-1})$. We denote this solution by $\Psi_k^{(1)}$; then the general solution to equation  (\ref{6.15}) is given by the formula
\begin{equation}\label{6.18}
\Phi_k^{(1)}=\Psi_k^{(1)}+\sum\limits_{s=1}^{m}b_{k,s}^{(1)}\phi_s,
\end{equation}
where $b_{k,s}^{(1)}$ are some constants. Having solved equation (\ref{6.15}), we return back to the boundary value problem (\ref{6.14}) and we find its solution by means of Lemma~\ref{lm6.1}.

In what follows we assume additionally that the eigenvalues of matrix $\mathrm{L}$ are different. Such assumption is technical and inessential; it is made just to simplify further calculations, see Remark~\ref{rm6.1} below.

The further process of solving boundary value problem (\ref{6.4}) for arbitrary $p$ is similar to above arguments. Namely, writing out solvability condition (\ref{6.6}) for problem (\ref{6.4}), we obtain equation for the function $\Phi_k^{(p-2)}(x')$ appearing while solving problem (\ref{6.4}) for $(p-2)$ as an arbitrary term in the general solution. Equation for function $\Phi_k^{(p-2)}$ is analogous to equation (\ref{6.15}) but from some other right hand side. The solvability condition of this equation, the orthogonality of the right hand side to all the eigenfunctions  $\phi_q$, $q=1,\ldots,m$, in $L_2(\mathds{R}^{n-1})$, allows us to determine  numbers $\L_k^{(p-2)}$. We then solve the obtained equation for $\Phi_k^{(p-2)}$ and by Lemma~\ref{lm6.1} we solve problem (\ref{6.4}). Functions  $\phi_k^{(p)}$ and numbers  $\L_k^{(p)}$ are described in the next statement.

\begin{lemma}\label{lm6.2}
There exist numbers $\L_k^{(p)}$, $p\geqslant 1$, such that the boundary value problem (\ref{6.4}) is solvable for each  $p\geqslant 1$. Solutions to these problems are represented as
\begin{equation}\label{6.21}
\phi_k^{(p)}(x',\xi)=\tilde{\phi}_k^{(p)}(x',\xi) +\check{\phi}_k^{(p)}(x',\xi) +\hat{\phi}_k^{(p)}(x',\xi) +\Phi_k^{(p)}(x'),
\end{equation}
where
\begin{align}
&\hat{\phi}_k^{(p)}=\mathcal{T}_6 \Phi_k^{(p-1)},\quad \check{\phi}_k^{(p)}=\mathcal{T}_7\Phi_k^{(p-2)},\label{6.22}
\\
&\Phi_k^{(p)}=\Psi_k^{(p)}+\sum\limits_{s=1}^{m} b_{k,s}^{(p)} \phi_s,\label{6.23}
\end{align}
$\Psi_k^{(p)}$ is the solution to the equation
\begin{align}
&(\Ho-\l^0)\Psi_k^{(p)}=h_k^{(p)}+\sum\limits_{q=1}^{p} \L_k^{(q)} \Phi_k^{(p-q)},\label{6.24}
\\
&h_k^{(p)}=\check{h}_k^{(p)}-\sum\limits_{s=1}^{m} b_{k,s}^{(p-1)} \int\limits_{-\frac{1}{2}}^{\frac{1}{2}} \left(
\frac{\p\hphantom{\nu}}{\p\nu^*}\frac{\p \hphantom{\xi}}{\p\xi} \mathcal{T}_7 + \mathcal{T}_5 \mathcal{T}_6
\right)\phi_s\di\xi,\nonumber
\\
&\check{h}_k^{(p)}:=-\int\limits_{-\frac{1}{2}}^{\frac{1}{2}} \frac{\p\hphantom{\nu}}{\p\nu^*}\frac{\p \hphantom{\xi}}{\p\xi} \Big(\tilde{\phi}_k^{(p+1)}+\mathcal{T}_7\Psi_k^{(p-1)}\Big)\di\xi 
-\int\limits_{-\frac{1}{2}}^{\frac{1}{2}} \mathcal{T}_5 \Big(\tilde{\phi}_k^{(p)}+\check{\phi}_k^{(p)}+\mathcal{T}_6 \Psi_k^{(p-1)}\Big)\di\xi,\nonumber
\end{align}
which is orthogonal to all eigenfunctions $\phi_q$, $q=1,\ldots,m$, in  $L_2(\mathds{R}^{n-1})$, and numbers $b_{k,s}^{(p)}$ and $\L_k^{(p)}$ are determined by the identities
\begin{equation}
\begin{aligned}
&\L_k^{(p)}=-(\check{h}_k^{(p)},\phi_k)_{L_2(\mathds{R}^{n-1})} - \sum\limits_{q=2}^{p-1} \L_k^{(q)} b_{k,q}^{(p-q)},
\\
&b_{k,k}^{(p)}=0,\quad b_{k,s}^{(p)}=\frac{ (\check{h}_k^{(p+1)},\phi_s)_{L_2(\mathds{R}^{n-1})} + \sum\limits_{q=2}^{p} \L_k^{(q)} b_{k,q}^{(p-q+1)}}{\L_{s}^{(1)}-\L_k^{(1)}}.
\end{aligned}\label{6.25}
\end{equation}
Functions $\tilde{\phi}_k^{(p)}$ are given by formula (\ref{6.8}) with
\begin{align*}
F=&\left(\frac{\p\hphantom{\nu}}{\p\nu^*}\frac{\p \hphantom{\xi}}{\p\xi}  - \frac{\p \hphantom{\xi}}{\p\xi} \frac{\p\hphantom{\nu}}{\p\nu}+ \mathcal{T}_5\right)  \Big(\tilde{\phi}_k^{(p-1)} + \check{\phi}_k^{(p-1)}\Big)\\
& -\l^0 \Big(\tilde{\phi}_k^{(p-1)}+ \check{\phi}_k^{(p-1)}+\hat{\phi}_k^{(p-1)}\Big)
+\sum\limits_{q=1}^{p-2} \L_k^{(q)}\phi_k^{(p-q-2)},
\\
g_-=&\frac{\p\hphantom{\nu}}{\p\nu} \Big(\tilde{\phi}_k^{(p-1)}+ \check{\phi}_k^{(p-1)}+\Phi_k^{(p-1)}\Big)\Big|_{\xi=-\frac{1}{2}}.
\end{align*}
Functions $\phi_k^{(p)}$ are infinitely differentiable w.r.t.  $x'$ and for each $\b\in\mathds{Z}_+^{n-1}$ the belongings
\begin{equation*}
\frac{\p^{|\b|} \phi_k^{(p)}}{\p x^\b} \in C^2(\overline{\Om})\cap L_\infty(\Om)
\end{equation*}
hold true.
\end{lemma}

\noindent The lemma can be proven easily by induction employing the expressions for functions $\phi_k^{(1)}$, $\phi_k^{(2)}$ obtained above. At that, one should assume that $\tilde{\phi}_k^{(1)}=\tilde{\phi}_k^{(2)}=\hat{\phi}_k^{(2)}=0$.

\begin{remark}\label{rm6.1}
The assumption on different eigenvalues for matrix $\mathrm{L}$ was employed in Lemma~\ref{lm6.2} for obtaining formulae  (\ref{6.25}). If this assumption does not hold, it just means that there is no complete splitting of leading terms in the asymptotics for  the perturbed eigenvalues. In this case it is a not a complicated problem to determine the terms of series (\ref{6.1}), (\ref{6.2}). The only difference is that on the next steps there appears a matrix similar to $\mathrm{L}$ which will determine the appropriate choice of eigenfunctions $\phi_k$. It imply no essential changes in the scheme of constructing the solutions to problem (\ref{6.4}).
\end{remark}

Thus, no matter how the eigenvalues of matrix $\mathrm{L}$ look like, it is possible to construct asymptotic series (\ref{6.1}), (\ref{6.2}) so that the next lemma holds true.

\begin{lemma}\label{lm6.3}
Let $N$ be an arbitrary natural number. The functions
\begin{equation*}
\phi_k^{(\e,N)}(x)=\e^{-1/2}\left(\phi_k(x')+\sum\limits_{p=1}^{N}
\e^p \phi_k^{(p)}(x',x_n\e^{-1})\right), \quad \l_k^{(\e,N)}=\l^0+\sum\limits_{p=1}^{N-2} \e^p \L_k^{(p)}
\end{equation*}
satisfy the estimates
\begin{align}
&\|\phi_k^{(\e,N)}-\e^{-1/2}\phi_k\|_{L_2(\Om^\e)}\leqslant C\e,\quad |\l_k^{(\e,N)}-\l^0|\leqslant C\e,\label{6.26}
\\
&
\|f_k^{(\e,N)}\|_{L_2(\Om^\e)}\leqslant C\e^{N-1},\quad
f_k^{(\e,N)}:=(\He-\l_k^{(\e,N)})\phi_k^{(\e,N)}.\label{6.27}
\end{align}
Here $C$ are some constants independent of $\e$ but depending, generally speaking, on $N$, while estimate (\ref{6.27}) involves the statement on belonging function $\phi_k^{(\e,N)}$ to domain of operator $\He$.
\end{lemma}

\section{Asymptotic expansions: justification}

In the present section we complete the proof of Theorem~\ref{th2.4} and justify the formal asymptotic expansions constructed in the previous section. First we prove two auxiliary statement and then we proceed to the justification. непосредственно обоснованием.

\begin{lemma}\label{lm7.1}
Eigenfunctions $\psi_k^\e$ of the perturbed operator can be chosen so that they satisfy the relations
\begin{equation}
(\psi_k^\e,\mathcal{P}\psi_j^\e)_{L_2(\Om^\e)}=\d_{jk},\quad j,k=1,\ldots,m. \label{7.2}
\end{equation}
\end{lemma}

\begin{proof}
In accordance with the results of the fifth section, each eigenfunction of the perturbed operator satisfy identities (\ref{5.18}) and (\ref{5.20}). Multiplying the eigenfunctions by appropriate constants, by (\ref{5.20}) we get (\ref{7.2}) for $j=k$. In view of these relations the form $(\,\cdot\,,\mathcal{P}\,\cdot)_{L_2(\Om^\e)}$ is a scalar product on an eigenspace of the perturbed operator associated with an eigenvalue. This is why these eigenfunctions can be chosen so that they satisfy relations (\ref{7.2}).

Suppose now eigenfunctions $\psi_k^\e$ and $\psi_j^\e$  are associated with different eigenvalues $\l_k^\e$ and $\l_j^\e$. Then taking into consideration the reality of these eigenvalues, by analogy with (\ref{5.19}) it is easy to check that
\begin{equation*}
0=\big((\He-\l_k^\e)\psi_k^\e,\mathcal{P}\psi_j^\e\big)_{L_2(\Om^\e)}= (\l_j^\e-\l_k^\e)\big(\psi_k^\e,\mathcal{P}\psi_j^\e\big)_{L_2(\Om^\e)},
\end{equation*}
that implies the desired identity (\ref{7.2}).
\end{proof}

\begin{lemma}\label{lm7.2}
For $\l$ in a small fixed neighborhood of point $\l^0$ and all sufficiently small $\e$ the resolvent $(\He-\l)^{-1}$ can be represented as
\begin{equation}\label{7.4}
(\He-\l)^{-1}=\sum\limits_{k=1}^{m} \frac{(\,\cdot\,, \mathcal{P}\psi_k^\e)_{L_2(\Om^\e)}}{\l_k^\e-\l}\psi_k^\e+\mathcal{T}_8^\e(\l),
\end{equation}
where operator $\mathcal{T}_8^\e(\l): L_2(\Om^\e)\to\H^1(\Om^\e)$ is bounded uniformly in  $\l$ and $\e$ and holomorphic w.r.t. $\l$, while functions $\psi_k^{(\e)}$ are chosen in accordance with Lemma~\ref{lm7.1}.
\end{lemma}

\begin{proof}
Let $\g$ be a circle of small radius centered at point $\l^0$ and containing no other points of the spectrum of operator $\Ho$.  Then by Theorem~\ref{th2.3} for sufficiently small $\e$ all the eigenvalues of perturbed operator converging to  $\l^0$ as $\e\to+0$ are located inside the circumference $\g$ and are separated from it by a positive distance. Now it follows from Theorem~\ref{th2.2} that the convergence
\begin{equation}\label{7.5}
-\frac{1}{2\pi\iu}\int\limits_{\g} (\He-\l)^{-1}\di\l\to
-\frac{1}{2\pi\iu}\int\limits_{\g} (\Ho-\l)^{-1}\oplus0\di\l
\end{equation}
holds true in the sense of norm of operator in  $L_2(\Om^\e)$. According to \cite[Гл. I\!I\!I, \S 6.5]{K}, both sides of this convergence are the projectors in $L_2(\Om^\e)$ and by the self-adjointness of operator $\Ho$ and \cite[Гл. V, \S 3.5]{K} it holds
\begin{equation}\label{7.6}
-\frac{1}{2\pi\iu}\int\limits_{\g} (\Ho-\l)^{-1}\oplus0\di\l=\sum\limits_{k=1}^{m} \frac{\e^{-1} (\cdot,\phi_k)_{L_2(\Om^\e)}}{\l^0-\l}\phi_k.
\end{equation}
According to \cite[Гл. I, \S 4.6]{K}, it implies that the dimension of the projector in the left hand side of   (\ref{7.5}) also equals $m$ for all sufficiently small $\e$.

The definition of eigenfunction yields immediately that
\begin{equation*}
(\He-\l)^{-1}\psi_k^\e=(\l_k^\e-\l)^{-1}\psi_k^\e,
\end{equation*}
and thus by \cite[Гл. I\!I\!I, \S 6.5, урав. (6.36)]{K} we have
\begin{equation*}
-\frac{1}{2\pi\iu}\int\limits_{\g} (\He-\l)^{-1}\psi_k^\e\di\l=\psi_k^\e.
\end{equation*}
Thus, the projector in the left hand side of (\ref{7.5}) is that on the finite dimensional space spanned over functions $\psi_k^\e$, $k=1,\ldots,m$. We stress that generally speaking it is not an operator of orthogonal projection, since operator $\He$ is non-self-adjoint. Thus,
\begin{equation}\label{7.7}
-\frac{1}{2\pi\iu}\int\limits_{\g} (\He-\l)^{-1}\di\l=\sum\limits_{k=1}^{m} c_k^\e(\cdot) \psi_k^\e,
\end{equation}
where $c_k^\e: L_2(\Om^\e)\to \mathds{C}$ are some functionals.

Let us determine functionals $c_k^\e$. For an arbitrary function $f\in L_2(\Om^\e)$ and $\l\in\g$ by analogy with (\ref{5.19}) we deduce
\begin{equation*}
(f,\mathcal{P}\psi_k^\e)_{L_2(\Om^\e)}=\big((\He-\l)u, \mathcal{P}\psi_k^\e\big)_{L_2(\Om^\e)} =(\l_k^\e-\l)(u,\mathcal{P}\psi_k^\e)_{L_2(\Om^\e)},
\end{equation*}
that by (\ref{7.7}) and Lemma~\ref{lm7.1} implies
\begin{equation}\label{7.8}
c_k^\e(f)= (f,\mathcal{P}\psi_k^\e)_{L_2(\Om^\e)}.
\end{equation}
In the sense of \cite[Гл. I\!I\!I, \S 6.5]{K}, to each eigenvalue $\l_k^\e$ an quasinilpotent operator is associated and it reads as
\begin{equation*}
-\frac{1}{2\pi\iu}(\He-\l_k^\e)\int\limits_{\g_k^\e} (\He-\l)^{-1}\di\l.
\end{equation*}
Here $\g_k^\e$ is a small circle centered at  $\l_k^\e$ containing no other eigenvalues of perturbed operator except $\l_k^\e$. Since the integral $-\frac{1}{2\pi\iu}\int\limits_{\g_k^\e} (\He-\l)^{-1}\di\l$ is a part of the corresponding projector in  (\ref{7.5}), by (\ref{7.7}), (\ref{7.8}), the aforementioned operator vanishes. Hence, by (\ref{7.7}), (\ref{7.8}) and \cite[Гл. I\!I\!I, \S 6.5, урав. (6.35)]{K} we get representation (\ref{7.4}), where $\mathcal{T}_8^\e(\l)$ is a bounded in $L_2(\Om^\e)$ operator holomorphic w.r.t. $\l$. It remains to prove that it is uniformly bounded in $\e$ and $\l$ and is holomorphic w.r.t. $\l$ as an operator from $L_2(\Om^\e)$ into $\H^1(\Om^\e)$.

For an arbitrary $f\in L_2(\Om^\e)$ by estimates (\ref{3.11}), the norm $\|\nabla(\He-\l)^{-1}f\|_{L_2(\Om^\e)}$ is uniformly estimated by the norms $\|f\|_{L_2(\Om^\e)}$ and $\|(\He-\l)^{-1}f\|_{L_2(\Om^\e)}$. By means of this estimate it is easy to check that operator $\mathcal{T}_8^\e(\l)$ is holomorphic w.r.t. $\l$ also as an operator from $L_2(\Om^\e)$ into $\H^1(\Om^\e)$.

By Theorem~\ref{th2.2}, as $\l\in\g$, the operator $(\He-\l)^{-1}$ converges to  $(\Ho-\l)^{-1}\oplus0$ as $\e\to+0$. Expressing operator $\mathcal{T}_8^\e(\l)$ by (\ref{7.4}), for $f\in L_2(\Om^\e)$, $\l\in\g$, and sufficiently small $\e$ we get the uniform estimate
\begin{equation}\label{7.9}
\|\mathcal{T}_8^\e(\l)f\|_{\H^1(\Om^\e)}\leqslant C\|f\|_{L_2(\Om^\e)},
\end{equation}
where $C$  is a constant independent of $f$, $\l$, and sufficiently small $\e$. By the modulus maximum principle for holomorphic function, estimate (\ref{7.9}) is valid also for  $\l$ lying inside circle $\g$. The proof is complete.
\end{proof}

We proceed to the justification. It follows from Lemmata~\ref{lm6.3},~\ref{lm7.2} that
\begin{equation}\label{7.10}
\phi_k^{(\e,N)}=\sum\limits_{q=1}^{m} \frac{\big(f_k^{(\e,N)}, \mathcal{P}\psi_q^\e\big)_{L_2(\Om^\e)}}{\l_q^\e-\l_k^{(\e,N)}}\psi_q^\e +\mathcal{T}_8^\e(\l_k^{(\e,N)})f_k^{(\e,N)},\quad k=1,\ldots,m.
\end{equation}
Employing Lemma~\ref{lm7.1}, we get
\begin{align}
&
\begin{aligned}
\Big(\phi_k^{(\e,N)}&-\mathcal{T}_8^\e(\l_k^{(\e,N)})f_k^{(\e,N)}, \mathcal{P}\big(\phi_j^{(\e,N)}-\mathcal{T}_8^\e(\l_j^{(\e,N)}) f_j^{(\e,N)}\big)\Big)_{L_2(\Om^\e)}
\\
&=\sum\limits_{q=1}^{m} \frac{\big(f_k^{(\e,N)},\mathcal{P}\psi_q^\e\big)_{L_2(\Om^\e)}} {\l_q^\e-\l_k^{(\e,N)}}
\overline{\left(\frac{\big(f_j^{(\e,N)},\mathcal{P}\psi_q^\e\big)_{L_2(\Om^\e)}} {\l_q^\e-\l_j^{(\e,N)}}
\right)}
\end{aligned}
\label{7.3}
\\
&\big(\phi_k^{(\e,N)}-\mathcal{T}_8^\e(\l_k^{(\e,N)})f_k^{(\e,N)}, \mathcal{P}\psi_j^\e\big)_{L_2(\Om^\e)}= \frac{\big(f_k^{(\e,N)},\mathcal{P}\psi_j^\e\big)_{L_2(\Om^\e)}} {\l_j^\e-\l_k^{(\e,N)}}.\label{7.1}
\end{align}
Lemmata~\ref{lm6.3},~\ref{lm7.2} yield the convergence
\begin{equation*}
\|\mathcal{T}_8^\e(\l_k^{(\e,N)})f_k^{(\e,N)}\|_{L_2(\Om^\e)}\to 0 \quad\text{as}\quad \e\to+0,
\end{equation*}
and estimates  (\ref{6.26}) follow the relations
\begin{align}
&\big(\phi_k^{(\e,N)},\mathcal{P}\phi_j^{(\e,N)}\big)_{L_2(\Om^\e)}=\d_{kj}+o(1),\quad \e\to+0,\nonumber
\\
&|F_{kj}^\e|\leqslant C,\quad F_{kj}^\e:=\frac{\big(f_k^{(\e,N)},\mathcal{P}\psi_j^\e\big)_{L_2(\Om^\e)}} {\l_j^\e-\l_k^{(\e,N)}},
\label{7.11}
\end{align}
where constant $C$ is independent of $\e$, $k$, $j$. It follows that the determinant of the matrix formed by the left hand sides of identities (\ref{7.3}) tends to one as  $\e\to+0$. On the other hand, the matrix formed by the right hand sides of identities (\ref{7.3}) can be represented as the product $\mathrm{F}^\e(\mathrm{F}^\e)^*$, where $\mathrm{F}^\e$ is the matrix with entries $F_{kj}^\e$. We thus get
\begin{equation*}
|\det \mathrm{F}^\e|\to 1\quad \text{при}\quad \e\to+0.
\end{equation*}
Therefore, for each sufficiently small $\e$ there exists permutation $q_1,\ldots,q_m$ such that
\begin{equation*}
\left|\prod\limits_{k=1}^{m} F_{kq_k}^\e\right|\geqslant \frac{1}{m!}.
\end{equation*}
By (\ref{7.11}) it implies
\begin{equation*}
|F_{kq_k}^\e|\geqslant \frac{1}{C^{m-1}m!}.
\end{equation*}
Substituting here the definition of $F_{kq_k}^\e$ in (\ref{7.11}) and employing estimates (\ref{6.27}), we arrive at the identities
\begin{equation*}
\l_{kq_k}^\e-\l_k^{(\e,N)}= \Odr(\e^{N-1}),
\end{equation*}
which prove asymptotics (\ref{1.22}) for the perturbed eigenvalues after an appropriate re-ordering.

We proceed to the justification of the asymptotics for the eigenfunctions. Suppose that condition (\ref{1.23}) is satisfied. Then it follows from asymptotics (\ref{1.22}) that
\begin{equation*}
|\l_j^\e-\l_k^{(\e,N)}|\geqslant C\e^r
\end{equation*}
as $N>r$, and hence by (\ref{6.27})
\begin{equation*}
|F_{kj}^\e|\leqslant C\e^{N-r-1},
\end{equation*}
where $C$ is a constant independent of $\e$, $k$, $j$. We substitute the latter estimates and (\ref{6.27}) into identity (\ref{7.10}) and move term $F_{kk}^\e\psi_k^\e$ into the left hand side. Then we obtain
\begin{equation*}
\|\phi_k^{(\e,N)}-F_{kk}^\e\psi_k^\e\|_{\H^1(\Om^\e)}\leqslant C\e^{N-r-1}.
\end{equation*}
Since $F_{kk}^\e$ is a number, $F_{kk}^\e\psi_k^\e$ is an eigenfunction associated with $\l_k^\e$. This is why the latter estimate proves asymptotics (\ref{1.24}) for the perturbed eigenfunctions. The proof of Theorem~\ref{th2.4} is complete.

\medskip

The author was partially supported by grant of RFBR, grant of President of Russia for young scientists-doctors of sciences (MD-183.2014.1) and Dynasty fellowship for young Russian mathematicians.


\bigskip
\begin{thebibliography}{99}


\bibitem{BB} C.M.~Bender, S.~Boettcher. {\it Real spectra in non-hermitian hamiltonians having PT symmetry}. Phys. Rev. Lett. {\bf 80}:24,  5243-5246 (1998).

\bibitem{M1}  A. Mostafazadeh. {\it Pseudo-Hermiticity  versus     PT-symmetry:     The necessary condition for the reality of the spectrum of a non-Hermitian Hamiltonian}. J. Math. Phys. {\bf 43}:1, 205-214 (2002).

\bibitem{M2}  A. Mostafazadeh. {\it Pseudo-Hermiticity versus PT-symmetry II: A complete characterization of non-Hermitian Hamiltonians with a real spectrum}. J. Math. Phys. {\bf 43}:54, 2814-2816 (2002).

\bibitem{M3} A. Mostafazadeh.  {\it Pseudo-Hermiticity versus PT-symmetry III: Equivalence of pseudo-Hermiticity and the presence of antilinear symmetries}. J. Math. Phys. {\bf 43}:8, 3944-3951 (2002).

\bibitem{M4} A. Mostafazadeh. {\it On the pseudo-Hermiticity of a class of PT-symmetric Hamiltonians in one dimension}.  Mod. Phys. Lett. A. {\bf 17}:30, 1973-1977 (2002).


\bibitem{Z1} M. Znojil. {\it Exact solution for Morse oscillator in PT-symmetric quantum mechanics}. Phys. Lett. A. {\bf 264}:2, 108-111 (1999).

\bibitem{Z2}
M. Znojil. {\it
Non-Hermitian matrix description of the PT-symmetric anharmonic oscillators}. J. Phys. A: Math. Gen. {\bf 32}:42, 7419-7428 (1999).



\bibitem{Z3} M. Znojil. {\it PT-symmetric harmonic oscillators}. Phys. Lett. A. {\bf 259}: 3-4, 220-223 (1999).

\bibitem{Z4}
G. Levai and M. Znojil. {\it
Systematic search for $\mathcal{PT}$-symmetric potentials with real energy spectra}. J. Phys. A: Math. Gen. {\bf 33}:40, 7165-7180 (2000).


\bibitem{B} С.М.~Bender. {\it Making sense of non-Hermitian Hamiltionians}. Rep. Prog. Phys. {\bf 70}:6, 947-1018 (2007).


\bibitem{Caliceti-Cannata-Graffi_2006}
E. Caliceti, F.~Cannata, S.~Graffi. {\it Perturbation theory of  $\mathcal{PT}$-symmetric Hamiltonians}. J. Phys. A. {\bf 39}:32,  10019-10027 (2006).

\bibitem{CGS}
E.~Caliceti, S.~Graffi, J.~Sj{\"o}strand. {\it   Spectra of $\mathcal{PT}$-symmetric operators and perturbation theory}. J.~Phys.~A. {\bf 38}:1, 185-193 (2005).


\bibitem{CCG2} E. Caliceti, F.~Cannata, S.~Graffi. {\it $\mathcal{PT}$-symmetric Schr\"odinger operators: reality of the perturbed eigenvalues}. SIGMA. {\bf 6}, id 009 (2010).

\bibitem{DDT}
P.~Dorey, C.~Dunning, and R.~Tateo. {\it Spectral equivalences, Bethe ansatz  equations, and reality properties in {$\mathcal{PT}$}-symmetric quantum
  mechanics}. J. Phys. A. {\bf 34}:28, 5679-5704 (2001).




\bibitem{KBZ} D. Krej\v{c}i\v{r}\'{\i}k, H.~B\'{\i}la and M.~Znojil. {\it Closed formula for the metric in the Hilbert space of a $\mathcal{PT}$-symmetric model}. J. Phys. A. {\bf 39}:32, 10143-10153 (2006).

\bibitem{Langer-Tretter_2004}
 H.~Langer and Ch.~Tretter. {\it A Krein space approach to PT-symmetry}. Czech. J.~Phys. {\bf 54}:10, 1113-1120 (2004).


\bibitem{Shin}
 K.C.~Shin. {\it On the reality of the eigenvalues for a class of $\mathcal{PT}$-symmetric oscillators}. Commun. Math. Phys. {\bf 220}:3, 543-564 (2002).

\bibitem{Znojil_2001}
M.~Znojil. {\it $\mathcal{PT}$-symmetric square well}. Phys. Lett.~A. {\bf 285}:1,2, 7-10 (2001).

\bibitem{BK1}  D. Borisov, D. Krejcirik. {\it $\mathcal{PT}$-symmetric waveguide}. Integr. Equat. Oper. Th. {\bf 62}:4, 489-515 (2008).


\bibitem{KT} D.~Krej\v{c}i\v{r}\'{\i}k  and M.~Tater. {\it
Non-Hermitian spectral effects in a PT-symmetric waveguide}.
J. Phys. A. {\bf 41}:24, id~244013 (2008).

\bibitem{KS} D.~Krej\v{c}i\v{r}\'{\i}k and P. Siegl. {\it
$\mathcal{PT}$-symmetric models in curved manifolds}.
J. Phys. A. {\bf 43}:48, id 485204 (2010).

\bibitem{IMM2012} D.I. Borisov. {\it On a $\mathcal{PT}$-symmetric waveguide with  a pair of small holes}. Trudy inst. matem. mekh. UrO RAN, {\bf 18}:2, 22-37 (2012). [Proc. Steklov Inst. Math. {\bf 281}:1 Supplement, 5-21 (2013).]

\bibitem{AsAn2012} D. Borisov, D.~Krej\v{c}i\v{r}\'{\i}k. {\it The effective Hamiltonian for thin layers with non-Hermitian Robin-type boundary conditions}. Asympt. Anal. {\bf 76}:1, 49-59 (2012).

\bibitem{N} S.A. Nazarov. {\it Variational and asymptotic methods for finding eigenvalues below the continuous spectrum threshold}. Sibir. matem. zhur. {\bf 51}:5, 1086-1101 (2010). [Siberian Math. J.  {\bf 51}:5, 866-878 (2010).]


\bibitem{G} R.R. Gadyl'shin. {\it On local perturbations of the Schr\"odinger operator on the plane}. Teor. matem. fiz. {\bf 132}:1, 97-104 (2002). [Theor. Math. Phys. {\bf 132}:1, 33-44 (2002).]



\bibitem{MSb} D. Borisov. {\it Discrete spectrum of a pair of non-symmetric waveguides coupled by a window}. Matem. sbornik. {\bf 197}:4, 3-32 (2006). [Sb. Math. {\bf 197}:4, 475-504 (2006).]


\bibitem{NazB} S.A. Nazarov.  {\it Asymptotic analysis of thin plates and rods. V. 1. Dimension reduction and integral estimates}. Nauchnaya kniga (IDMI), Novosibirsk (2002). (in Russian).

\bibitem{AHP2009}  D. Borisov and P. Freitas. {\it Singular asymptotic expansions for Dirichlet eigenvalues and eigenfunctions on thin planar domains}. Ann. Inst. H. Poincar\'e Anal. Non Lin\'eaire. {\bf 26}:2, 547-560 (2009).


\bibitem{JFA2010} D. Borisov, and P. Freitas. {\it Asymptotics of Dirichlet eigenvalues and eigenfunctions of the Laplacian on thin domains in $\mathds{R}^d$}. J.  Funct. Anal. {\bf 258}:3,  893-912 (2010).

\bibitem{ESAIM} D. Borisov, and G. Cardone. {\it Complete asymptotic expansions for the eigenvalues of the Dirichlet Laplacian in thin three-dimensional rods}.  ESAIM. Contr. Op. Ca. Va. {\bf 17}:3, 887-908 (2011).

\bibitem{CCP} G.~Cardone, A.~Corbo-Esposito,  G.~Panasenko. {\it Asymptotic partial decomposition for
diffusion with sorption in thin structures}. Nonlin. Anal. {\bf 65}:1, 79-106 (2006).

\bibitem{PP} G.~Panasenko, E.~Perez. {\it Asymptotic partial decomposition of domain for spectral problems in rod structures}.  J. Math. Pures Appl. {\bf 87}:1, 1-36 (2007).


\bibitem{Izv2003} D.I. Borisov. {\it  Asymptotics and estimates for the eigenelements of the Laplacian with frequently alternating nonperiodic boundary conditions}. Izv. RAN. Ser. matem. {\bf 67}:6,  23-70 (2003). [Izv. Math. {\bf 67}:6, 1101-1148 (2003).]


\bibitem{MS}  N.N. Bogoliubov and Y.A. Mitropolski. {\it Asymptotic methods in the theory of non-linear oscillations}. Nauka, Moscow (1974). [Gordon and Breach, New York (1961).]



\bibitem{K} T. Kato. {\it Perturbation theory for linear operators}. Berlin, Springer, 1995.

\bibitem{Mi} V.P. Mikhajlov. {\it Partial differential equations.} Mir, Nauka (1978).

\bibitem{Ld} O.A. Ladyzhenskaya, N.N. Uralceva.  {\it  Linear and quasilinear elliptic equations}. Nauka, Moscow (1973). [New York, Academic Press (1968).]
\bibitem{MPAG07} D. Borisov. {\it Asymptotic behaviour of the spectrum of a waveguide with distant perturbation}. Math. Phys. Anal. Geom. {\bf 10}:2, 155-196 (2007).
    \end{thebibliography}
\end{document}